\documentclass[10pt,reqno]{amsart}
\usepackage{amsmath}
\usepackage{amsfonts}
\usepackage{amssymb}
\usepackage{graphicx}
\usepackage{amsthm,graphicx,color,yfonts}
\usepackage{pdfsync}
\usepackage{epstopdf}%
\usepackage{pifont}
\usepackage{bbm}
\usepackage{enumitem}
\usepackage[letterpaper, margin=4.1cm]{geometry}

\usepackage[colorlinks=true]{hyperref}
\hypersetup{linkcolor=red,citecolor=blue,filecolor=dullmagenta,urlcolor=blue} 

\usepackage{wrapfig}
\usepackage{tikz}
\usetikzlibrary{arrows,calc,decorations.pathreplacing}
\definecolor{light-gray1}{gray}{0.90}
\definecolor{light-gray2}{gray}{0.80}
\definecolor{light-gray3}{gray}{0.60}

\newcommand{\wt}{\widetilde}

\newcommand{\vp}{\varphi}

\newcommand{\R} {\mathbb R}
\newcommand{\cuad}{{\sqcap\kern-.68em\sqcup}}

\newcommand{\ve}{\varepsilon}

\newcommand{\be}{\begin{equation}}
\newcommand{\ee}{\end{equation}}

\newcommand{\la}{\lambda(t)}

\definecolor{darkgreen}{rgb}{0.2,0.7,0.1}

\newcommand{\sech}{\mathop{\mbox{\normalfont sech}}\nolimits}

\newcommand{\px}{\partial_x}
\newcommand{\pt}{\partial_t}
\newcommand{\nlop}{(1-\partial_x^2)^{-1}}

\newcommand{\al}{\alpha}
\newcommand{\bt}{\beta}

\def\bm{\left( \begin{array}{cc}}
\def\endm{\end{array}\right)}

\providecommand{\norm}[1]{\left\| #1 \right\|}

\newcommand{\ba}{\begin{equation*}}
\newcommand{\ea}{\begin{equation*}}
\newcommand{\bea}{\begin{eqnarray}}
\newcommand{\eea}{\end{eqnarray}}
\newcommand{\bee}{\begin{eqnarray*}}
\newcommand{\eee}{\end{eqnarray*}}
\newcommand{\ben}{\begin{enumerate}}
\newcommand{\een}{\end{enumerate}}

\numberwithin{equation}{section}

\newtheorem{theorem}{Theorem}[section]
\newtheorem*{theorem*}{Theorem}
\newtheorem{proposition}{Proposition}[section]
\newtheorem{corollary}{Corollary}[section]
\newtheorem{lemma}{Lemma}[section]
\newtheorem{definition}{Definition}[section]

\theoremstyle{remark}

\title[Decay in abcd systems]{Decay of small energy solutions in the ABCD Boussinesq model under the influence of an uneven bottom}

\author[Maul\'en]{Christopher Maul\'en}  
\address{Fakultat f\"ur Mathematik, Universit\"at Bielefeld,  Postfach 10 01 31, 33501 Bielefeld, Germany.}
\email{cmaulen@math.uni-bielefeld.de}
\thanks{Ch.Ma. was funded by the Deutsche Forschungsgemeinschaft (DFG, German Research Foundation) -- Project-ID 317210226 -- SFB 1283, and Chilean research grants ANID Exploration project 13220060 and FONDECYT 1231250.}

\author[Mu\~noz]{Claudio Mu\~noz}
\address{CNRS and Departamento de Ingenier\'{\i}a Matem\'atica and Centro
de Modelamiento Matem\'atico (UMI 2807 CNRS), Universidad de Chile, Casilla
170 Correo 3, Santiago, Chile.}
\email{cmunoz@dim.uchile.cl}
\thanks{C. M.'s work was partly funded by Inria Lille PANDA and Chilean research grants ANID Exploration project 13220060, FONDECYT 1231250 and Centro de Modelamiento Matemático (CMM) BASAL fund FB210005 for center of excellence from ANID-Chile.}

\author[Poblete]{Felipe Poblete}
\address{Instituto de Ciencias F\'isicas y Matem\'aticas, Facultad de Ciencias, Universidad Austral de Chile, Valdivia, Chile.}
\email{felipe.poblete@uach.cl}
\thanks{F.P.'s work was partially supported by by Inria Lille PANDA , ANID Exploration project 13220060, and ANID project FONDECYT 1221076.}


\begin{document}

\begin{abstract}
The $abcd$ Boussinesq system, introduced by Bona, Chen, and Saut, describes a four-parameter $(a,b,c,d)$ family of models formulated on the time-space domain $\mathbb{R}_t \times \mathbb{R}_x$. It serves as a first-order two-wave approximation to the two-dimensional incompressible, irrotational water wave equations in shallow water, inspired by Boussinesq’s classical derivation. Within the different parameter regimes, the generic regime is described by $b,d>0$ and $a,c<0$ while the system becomes Hamiltonian when $b=d$. Previously, sharp local in space $H^1\times H^1$ decay properties were proved in the case of a large class of $abcd$ model under the small data assumption. In this paper, we generalize [C. Kwak, \emph{et. al.}, \emph{The scattering problem for Hamiltonian ABCD Boussinesq systems in the energy space}. J. Math. Pures Appl. (9) 127 (2019), 121--159] by considering the small data $abcd$ decay problem in the physically relevant \emph{variable bottom regime} described by M. Chen. The nontrivial bathymetry is represented by a smooth space-time dependent function $h=h(t,x)$, which obeys integrability in time and smallness in space. We prove first the existence of small global solutions in $H^1\times H^1$.  Then, for a sharp set of dispersive $abcd$ systems (characterized only in terms of parameters $a, b$ and $c$), every $H^1\times H^1$ small solution must converges to zero inside of  the light cone $|x|\leq |t|$. 
\end{abstract}

\maketitle

\section{Introduction and Main Results}

\subsection{Problem formulation} In this work we investigate  the long time dynamical properties for solutions of the one-dimensional $abcd$ system in the presence of variable bathymetry. Seeking to consider the more realistic variable bottom case, but still under a rigorous mathematical treatment, we present here the first installment of a series of papers dealing with this issue, of key interest for engineers, geophysicists and other scientists. The variable bottom abcd model under consideration has been derived by M. Chen \cite{MChen} and reads as follows: for $(t,x)\in \R\times\R,$
\begin{equation}\label{boussinesq_0}
\begin{cases}
(1- b\,\partial_x^2)\partial_t \eta  + \partial_x\!\left( a\, \partial_x^2 u +u + (\eta + h) u \right) =  (-1 + \tilde a_1 \partial_x^2) \partial_th  , \\
(1- d\,\partial_x^2)\partial_t u  + \partial_x\! \left( c\, \partial_x^2 \eta + \eta  + \frac12 u^2 \right) =  \tilde  c_1  \partial_t^2 \partial_x h.
\end{cases}
\end{equation}
Here $u=u(t,x)$ and $\eta=\eta(t,x)$ are real-valued scalar functions, representing a component of the velocity and the surface in a shallow water free surface, inviscid and irrotational Boussinesq's regime. The variable bottom is represented by the function $h=h(t,x)$, perturbation of the flat bottom which is recovered by setting $h\equiv 0$. 
Along this paper the function $h$ will be assumed as a given data of the problem, having dependence on time and space. Additionally, the parameters $\tilde a_1$ and $\tilde c_1$ measure the magnitude of time/space variations of the uneven bottom in the model, and will be also assumed known along this work. We have chosen for simplicity from \cite{MChen} as reference variation of pressure  $(P_0)_x=0$.

In the case where one has a fixed bottom, the equation was  originally derived by Bona, Chen and Saut \cite{BCS1,BCS2} as a first order, one dimensional asymptotic regime model of the water waves equation, in the vein of the Boussinesq original derivation \cite{Bous}, but maintaining all possible equivalences between the involved physical variables, and taking into account the shallow water regime. The parameters of physical perturbations considered for the expansion are
\[
\al := \frac Ah \ll 1, \quad \bt := \frac{h^2}{\ell^2} \ll 1, \quad \al \sim \bt,
\]
and where $A$ and $\ell$ denote the  typical amplitude and wavelength of the waves, and $h$ represents the constant depth. Equations \eqref{boussinesq_0} belong to a broader class of Boussinesq models, which also include second-order formulations, and these were developed in \cite{BCS1} and \cite{BCL}.
Equations \eqref{boussinesq_0} are part of a hierarchy of Boussinesq models, including second order systems, which were obtained in \cite{BCS1} and \cite{BCL}. A rigorous derivation of model \eqref{boussinesq_0} from the full free surface Euler equations (as well as extensions to higher dimensions) was given by Bona, Colin and Lannes \cite{BCL}, see also Alvarez-Samaniego and Lannes \cite{ASL} for refined results. Since that time, these models have received extensive attention in the literature; see, for instance, \cite{BLS,BCL,LPS,Saut} and references therein for a detailed account of known results. This list is not intended to be exhaustive.

The constants in \eqref{boussinesq_0} satisfy structural constraints  following the conditions \cite{BCS1}
\be\label{Conds0}
\begin{aligned}
& a=\frac12 \left(\theta^2-\frac13\right)\lambda, \quad b=\frac12 \left(\theta^2-\frac13\right)(1-\lambda), \\
& c=\frac12 \left(1- \theta^2 \right)\mu, \quad d=\frac12 \left(1-\theta^2\right)(1-\mu), \\
& \tilde a_1 =\frac12 \left((1-\lambda)\left(\theta^2-\frac13\right)+1-2\theta\right), \quad \tilde c_1=1-\theta, \\
\end{aligned}
\ee
for some $\theta\in [0,1]$ and $\lambda,\mu\in\mathbb R$. Notice that  $a+b =\frac12\left(\theta^2-\frac13\right)$ and $c+d =\frac12 (1-\theta^2)\geq 0$. Moreover, $a+b+c+d = \frac13$ is independent of $\theta$. (This case is referred to as the regime without surface tension $\tau\geq 0$, otherwise one has parameters $(a,b,c,d)$ such that $a+b+c+d = \frac13-\tau$.)

In a previous work \cite{KMPP},  the flat bottom $h\equiv 0$ and \emph{Hamiltonian generic} case \cite[p. 932]{BCS2}
\be\label{Conds2}
b,d >0, \quad a,~c<0, \quad b=d,
\ee
was considered. Indeed, always in the flat bottom case, under \eqref{Conds2} it is well-known \cite{BCS2} that \eqref{boussinesq_0} is globally well-posed in $H^s \times H^s$, $s\geq 1$, for small data, thanks to the preservation of the energy/hamiltonian
\be\label{Energy}
H[\eta,u ](t):= \frac12\int \left( -a (\partial_x u)^2 -c (\partial_x \eta)^2  + u^2+ \eta^2 + u^2\eta \right)(t,x)dx,
\ee
which is conserved by the flow. From now on, we will identify $H^1 \times H^1$ as the standard \emph{energy space} for \eqref{boussinesq_0}. Under the additional condition $3b( a+  c) + 2b^2  < 8 a  c$, decay of small solutions was shown, in the sense that
\begin{equation}\label{decay0}
\lim_{t \to \pm \infty} \|(\eta,u)(t)\|_{H^1\times H^1 (I(t))} =0,\quad I(t)= \left(-\frac{|t|}{\log^2|t|},\frac{|t|}{\log^2|t|} \right).
\end{equation}
 An improvement was obtained in \cite{KM20}, in the sense that more $abcd$ models were proved to follow \eqref{decay0} under less demanding conditions on $(a,b,c)$, see Definition \ref{New Dis_Par} for the detailed new conditions. 
 
 A key element in previous proof is the Hamiltonian character of the model, and the conservation of a macroscopic energy. However, in our case the energy \eqref{Energy} is not conserved anymore, unless the bottom is time independent. Indeed, consider the modified energy for the variable bottom case
\be\label{Energy_new}
H_h[\eta,u](t):= \frac12\int \left( -a (\partial_x u)^2 -c (\partial_x \eta)^2  + u^2+ \eta^2 + u^2(\eta + h) \right)(t,x)dx.
\ee
Then at least formally the following is satisfied: from \eqref{Energy_new} and \eqref{boussinesq_0},
\be\label{Energy_new_00}
\begin{aligned}
\frac{d}{dt} H_h[u ,\eta ](t) = &~{} -a c_1 \int  u  \partial_t^2 \partial_x h +c_1 \int \left( 1+a    +   \eta + h  \right)u (1- \partial_x^2)^{-1}  \partial_t^2 \partial_x h \\
&~{}  + c\int \eta (1 - a_1 \partial_x^2) \partial_th    \\
&~{} +(a_1- 1) \int \left(  (1+c) \eta  + \frac12 u^2 \right)(1- \partial_x^2)^{-1}  \partial_th  \\
&~{}  -a_1 \int \left(  (1+c) \eta  + \frac12 u^2 \right) \partial_th  + \frac12 \int u^2  \partial_t h.
\end{aligned}
\ee
Notice that the influence of the uneven bottom is important in the long time behavior through the previous identity, meaning that the previous results proved in \cite{KMPP} or \cite{KM19} are not easily translated to the uneven bottom regime. No particular sign condition or similar property seems to help to control the output in \eqref{Energy_new_00}. In that sense, we will choose data and conditions on $h$ that ensure globally well-defined solutions with bounded in time energy, a naturally physical condition.

 Assume \eqref{Conds2}. By considering the stretching of variables given by $ u(t/\sqrt{b},x/\sqrt{b})$, $\eta(t/\sqrt{b},x/\sqrt{b})$, and $h(t/\sqrt{b},x/\sqrt{b})$, $a\to a/b$ and $c\to c/b$, we can assume $b=d=1$ and rewrite \eqref{boussinesq_0} as the slightly simplified model
\begin{equation}\label{boussinesq}
\begin{cases}
(1- \partial_x^2)\partial_t \eta  + \partial_x \! \left( a\, \partial_x^2 u +u + (\eta+h) u  \right) =(-1 + a_1 \partial_x^2) \partial_th  , \\
(1- \partial_x^2)\partial_t u  + \partial_x \! \left( c\, \partial_x^2 \eta + \eta  + \frac12 u^2 \right) = c_1  \partial_t^2 \partial_x h,
\end{cases}
\end{equation}
for $(t,x)\in \R\times\R$ and with new conditions (see \eqref{Conds0}-\eqref{Conds2})
\be\label{Conds}
\begin{aligned}
&a,c<0, \quad a +1 =\frac1{2b}\left(\theta^2-\frac13\right),\quad c+1 =\frac1{2b} (1-\theta^2)\geq 0,\\
& a_1=\frac{\tilde{a}_1}{b},\quad c_1=\frac{\tilde{c}_1}{b},
\end{aligned}
\ee
for some $\theta\in [0,1].$  Note that in particular now $c>-1$, and equation \eqref{boussinesq} will be the main subject of this paper.

Compared with the original $abcd$ model, equation \eqref{boussinesq} is characterized by three additional features not present in the flat bottom case: first, the surface $\eta$ is modified by $h$ as part of the nonlinear terms, leading to the most prominent perturbation of the model, and characterizing local and global solutions. Then, two source terms appear, characterized by an important amount of derivatives of the variable bottom. These terms must be understood as perturbations of the original model, and in the long time behavior will represent perturbations of virial terms that need care and control. 

\subsection{Main results}
In this work, we extend the results of \cite{KMPP} for the scalar good Boussinesq model by analyzing the more physically relevant and mathematically challenging abcd Boussinesq system, under the conditions previously described. By employing suitably constructed virial identities, we show that every sufficiently small global solution in $H^1\times H^1$ solutions must decay in time in time-growing intervals in space, remaining inside a restricted portion of the light cone. To express the result precisely, we first introduce a dispersive condition formulated in terms of the system’s parameters.

\medskip

\noindent
{\bf Hypotheses on $a,c$}. Recall that, in the considered Hamiltonian setting, $b=d>0.$  Note that from \eqref{Conds} we have $-1\leq c<0$ and $-1 - \frac{1}{6b} \leq  a <0$. We shall assume the sharp values of the parameters $(a,b,c)$ obtained in \cite{KM20}. 

\begin{definition}[Refined dispersion-like parameters]\label{New Dis_Par}
We say that $(a,b,c)$ satisfying  \eqref{Conds} are \emph{refined dispersion-like parameters} if 
\be\label{dispersion_like_b}
3( a+  c) + 2  < 8 a  c, 
\ee
or they satisfy the following conditions:
\ben
\item either
\begin{equation}\label{ref_dispersion_like2-1}
\begin{cases}\begin{array}{lcl} 
 45 a c >1-  a , &\;&\mbox{for} \quad  -1 \leq  c < -\frac{1}{90}(19+\sqrt{181}), \\ 
18 a  c + a+ c>0, &\;&\mbox{for} \quad -\frac{1}{90} (19+\sqrt{181})\le  c < -\frac{1}3,\\
 27 a c>6 a + 1 , &\;&\mbox{for} \quad -\frac{1}3 \le  c< -\frac{1}9, 
\end{array}\end{cases}
\end{equation}
whenever $ c \le  a < 0$,  
\item or,
\begin{equation}\label{ref_dispersion_like2-2}
\begin{cases}\begin{array}{lcl} 
45 a  c>1-  c , &\;&\mbox{for} \quad -1-\frac{1}{6b} \leq  a < -\frac{1}{90} (19+\sqrt{181}),\\
18 a c +  a+  c>0, &\;&\mbox{for} \quad -\frac{1}{90} (19+\sqrt{181})\le  a < -\frac{1}3,\\
 27 a c>6 c + 1 , &\;&\mbox{for} \quad -\frac{1}3 \le  a< -\frac{1}9,
\end{array}\end{cases}
\end{equation}
whenever $ a \le  c < 0$.
\een
\end{definition}

The conditions \eqref{dispersion_like_b}-\eqref{ref_dispersion_like2-1}-\eqref{ref_dispersion_like2-2} formally say that $(a,c)$, being both negative, need to be sufficiently far from zero. Indeed, for values of $a$ and $c$ sufficiently close to zero,  \eqref{dispersion_like_b} is not satisfied. From the physical point of view, this condition can be understood as having \emph{sufficient dispersion} in \eqref{boussinesq} to allow decay to zero for both components of the flow.  

In the flat bottom case, Definition \ref{New Dis_Par} is sharp in the following sense: below $b=\frac3{16}$, i.e. above $a=c=-\frac{1}{48}$, the \emph{group velocity} $v=\omega'(k)$ of plane waves associated with the linear $abcd$ system  is allowed to vanish at a nonzero wavenumber $k$, implying that scattering may not occur following a standard procedure \cite{KM20}. We shall see that in the moving bottom case additional terms appear so that Definition \ref{New Dis_Par} will become a sort of necessary condition to have local decay. 

\medskip

\noindent
{\bf Hypotheses on the moving bottom $h$}. Let $\varepsilon_0>0$ be a small parameter to be fixed below and assume $\varepsilon\in(0,\varepsilon_0).$ We shall assume that the bottom $h$ satisfies, for some constant $C>0,$
\begin{equation}\label{hyp_h}
\begin{aligned}
& h: [0,\infty)\times \mathbb R \to \mathbb R, 
\\
&  \|h \|_{W^{2,\infty}_t H^1_x}  + \| \partial_t h\|_{L^1_{t}H^1_x}  +\| \partial_t^2 h\|_{L^1_{t}H^1_x} \leq C \varepsilon, \\
& \| \partial_x h \|_{L^1_t L^\infty_x} \leq C.
\end{aligned}
\end{equation}
Notice that in \eqref{hyp_h} smallness is required to allow global small solutions. Additionally, the $L^1_t$ integrability condition on $\|\partial_t^k h\|_{H^1}$, $k=1,2$, seems demanding but it is a standard condition for models with forcing terms. In particular, the forcing in \eqref{boussinesq} has time and space derivatives. Below we will review the literature in this respect, and will find that usually in applications the hypotheses on $h$ are reasonable in the same direction as the ones described above: locally defined in time, absence after some time, local spatial behavior of perturbations, etc.

Let $I(t)$ be defined as the interval
\begin{equation}\label{I(t)}
 I(t):= \left(- \frac{ |t|}{\log |t| \log^2 ( \log |t|)}, \frac{ |t|}{\log |t| \log^2 (\log |t|)}\right),  \quad |t|\geq 11.
\end{equation}
Our main result is the existence and long-time dynamics of small solutions to the variable bottom Hamiltonian $abcd$:

\begin{theorem}[Global existence and decay]\label{MT1}
There exist $C_0,\varepsilon_0>0$ such that, for all $0<\varepsilon<\varepsilon_0$, the following is satisfied. Let $(a,c)$ be parameters such that  Definition \ref{New Dis_Par} is satisfied. Let $h$ be a moving bottom satisfying the hypotheses \eqref{hyp_h}, and consider initial data such that  
\be\label{Smallness}
\|(u_0,\eta_0)\|_{H^1\times H^1}< \ve.
\ee
Then, there exists a unique global solution $(u,\eta)\in C(\R, H^1\times H^1)$ of \eqref{boussinesq} from initial data at time zero $(u_0,\eta_0)$ satisfying
\begin{equation}\label{global}
\sup_{t\geq0}  \|(u,\eta)(t)\|_{H^1\times H^1}  \le C_0 \ve.
\end{equation}
Moreover, for $I(t)$ as in \eqref{I(t)}, there is strong decay:
\be\label{Conclusion_0}
\lim_{t \to \infty}   \|(u,\eta)(t)\|_{(H^1\times H^1)(I(t))} =0.
\ee
\end{theorem}

Theorem \ref{MT1} is, as far as we know, the first decay-in-time result for the one dimensional $abcd$ Boussinesq systems in the context of variable bottom. Previously, in the flat bottom case $h=0$, Mu\~noz and Rivas \cite{Munoz_Rivas} proved decay at a rate of order $O(t^{-1/3})$ for data in weighted Sobolev spaces. This was done for the case of a generalized $abcd$ system with $a=c$, $b=d$ and nonlinearities $\partial_x(u^p \eta)$ and $\partial_x(u^{p+1}/(p+1))$, $p\geq 5$. However, equations \eqref{boussinesq} (characterized by the exponent $p=1$) involve important challenging issues: \emph{long range} (quadratic) nonlinearities in 1D, very weak linear dispersion $O(t^{-1/3})$, the existence of (non scattering) solitary waves, standard enemies to decay, and finally the variable bottom that makes the problem non Hamiltonian, with possible blow up solutions.

Theorem \ref{MT1} also provides a corollary discarding the existence of solitary waves in the region $I(t)$.

\begin{corollary}
Under the assumptions in Theorem \ref{MT1}, there are no small solitary waves nor breathers present in the region $I(t)$ as $t\to +\infty.$
\end{corollary}

It is well-known that flat bottom $abcd$ has solitary waves, some of them stable. See the works \cite{BCL0,CNS1,CNS2,HSS,Olivera} and references therein for further details. In the variable bottom case, the interaction of solitary waves and the mild bottom has been recently considered in \cite{DGMMP}.

Theorem \ref{MT1} provides a mild integrated rate of decay for solutions in the energy space: 
\be\label{eq:virial1_intro}
\int_{11}^\infty \!\! \int e^{-c_0 |x|} \left(u^2 + \eta^2 + (\partial_x u)^2+ (\partial_x \eta)^2\right)(t,x)dx\, dt  \lesssim_{h} 1.
\ee
This reveals that local-in-space $H^1$ norms do integrate in time. Formally, a better rate of decay is present, in the form of a local smoothing effect. This decay estimate can be compared to the expected linear decay rate, which is generally believed to be $O_{L^\infty}(t^{-1/3})$ \cite{Munoz_Rivas} (recall that not all $a,b,c$ lead to linear systems with the same rate of decay). In that sense, \eqref{eq:virial1_intro} reveals that a hidden improved decay mechanism is present in \eqref{boussinesq}. As far as we understand, no better estimate can be obtained from our approach, unless some additional extra assumptions on the decay of the initial data are imposed. It is also worth highlighting that the linear decay rate naturally leads (at least at a formal level) to the phenomenon of modified scattering—or what is sometimes termed “supercritical scattering” due to the long-range nonlinear effects. See \cite{KMM1,KMM3} for a complete discussion on this topic.

\subsection{Related literature} As for the low regularity Cauchy problem associated to \eqref{boussinesq_0} and its generalizations to higher dimensions, Saut et. al. \cite{SX,SWX} studied in great detail the long time existence problem by focusing in the small physical parameter $\ve$ appearing from the asymptotic expansions. They showed well-posedness (on a time interval of order $1/\ve$) for \eqref{boussinesq_0}. Previous results by Schonbek \cite{Schonbek} and Amick \cite{Amick} considered the case $a=c=d=0$, $b=\frac13$, a version of the original Boussinesq system, proving global well-posedness under a non cavitation condition, and via parabolic regularization. The so-called KdV-KdV and BBM-BBM are naturally good candidates for a deeper study. Indeed, Linares, Pilod and Saut \cite{LPS} considered existence in higher dimensions for time of order $O(\ve^{-1/2})$, in the KdV-KdV regime $(b=d=0)$. Additional low-regularity well-posedness results can be found in the recent work by Burtea \cite{Burtea}. On the other hand, ill-posedness results and blow-up criteria (for the case $b=1$, $a=c=0$), are proved in \cite{CL}, following the ideas in Bona-Tzvetkov \cite{BT}. Furthermore, Kwak and Maul\'en \cite{KM}, following the ideas of Bejenaru and Tao \cite{BTao}, proved ill-posedness in the generic ($a,c<0$, $b,d>0$), KdV-KdV and BBM-BBM regime ($a=c=0$).

The study of Boussinesq systems over variable bottom topography has been under strong consideration during past years. Peregrine~\cite{Peregrine} derived the vector system
\begin{equation}\label{Pere}
\begin{aligned}
& \partial_t \eta + \nabla \cdot \left[(h + \varepsilon \eta) \overline{\mathbf{u}}\right] = 0, \\
& \partial_t \overline{\mathbf{u}} + \nabla \eta + \varepsilon(\overline{\mathbf{u}} \cdot \nabla)\overline{\mathbf{u}} - \sigma^2 \frac{h}{2} \nabla\left(\nabla \cdot \left(h \partial_t \overline{\mathbf{u}}\right)\right) + \sigma^2 \frac{h^2}{6} \nabla\left(\nabla \cdot \partial_t \overline{\mathbf{u}}\right) = O\left(\varepsilon \sigma^2, \sigma^4\right).
\end{aligned}
\end{equation}
Here, $\overline{\mathbf{u}}$ denotes the depth-averaged horizontal velocity, defined as
\begin{equation*}
\overline{\mathbf{u}} = \frac{1}{h + \varepsilon \eta} \int_{-h}^{\varepsilon \eta} \mathbf{u} \, dz.
\end{equation*}
Several variants of Boussinesq-type systems incorporating bottom variability have been developed in the literature~\cite{Chazel, Dutykh, Madsen, Madsen2, Nwogu, Schaffer}. Precisely, we have decided to understand the long time behavior of small solutions in the variable bottom case described by M. Chen \cite{MChen}, but several other models, specially the Peregrine one are also expected to be treated in the near future. Compared with \eqref{boussinesq}, one sees that \eqref{Pere} is vector valued but a similar structure is present. In some sense, we believe that techniques used in this paper may be useful to treat the more involved case \eqref{Pere}. Another interesting variable bottom model was derived by Mitsotakis \cite{Mitso1}. Compared with the model worked in this paper, Mitsotakis' model generalizes the Peregrine model and considers a more complex dynamics. In that sense, it is a nice candidate to consider the long time behavior problem once the Peregrine model is well-understood.

There are some recent advances on the understanding on the $abcd$-Boussinesq model with variable bottom. For example, in \cite{CGL}, the authors consider the equation \eqref{boussinesq_0}, with the respective pressure term $\partial_x P_0$, and a bottom topography defined by $h(t,x)=\chi(t) B(x)$ where $\chi$ is a Heaviside function. This setting is commonly used to describe sudden bottom displacement, such as those occurring during  underwater earthquakes, which may generate tsunamis. Usually this problem is referred to as the “Bump" problem. In particular, the authors studied the case $P_0=0$, $b=-1/3$, $a=c=\tilde{c}_1=0$ and $\tilde{a}_1=0$, and they discussed convergence results for a particular regularized version of the “Bump” problem, approximating an instantaneous bottom lift-up.

The extension to the 2D version of the model \eqref{boussinesq_0}  has been widely explored for its viability using computational methods. 
Kervella, Dutykh and Dias \cite{KDD}, in the context of tsunami generation by underwater earthquakes, derived the extended version of \eqref{boussinesq_0} to the 2D case, with the parameters $a=-1/6$, $b=1/2$, $d=c=0$, $\tilde{c}_1=0$ and $\tilde{a}_1=0$. Their model incorporates sea bottom motion caused by dip–slip faults and describes the  resulting motion of the water above (see (54)-(56) in \cite{KDD}).  
They discussed the relevance of the dispersive effects on this model and analyzed the interpretation made in \cite{Indian} about the satellite altimetry observations of the Indian Ocean tsunami (26 December 2004). Related to this phenomenon, Chen \cite{Chen_sim} developed numerical simulations using the 2D BBM-BBM \eqref{boussinesq_0}, and based on data of this destructive event, she remarks that in many real physical situations, the wave is generated by a source which is not necessarily axisymmetric.

 Dutykh and Dias  \cite{DD} employed the 2D nonlinear shallow water equation to model tsunami generation, leading to the model \eqref{boussinesq_0},  after a suitable redefinition of the bottom function $h$,  with all parameters set to zero (see equation (6) in \cite{DD}). This case can be interpreted as corresponding to a surface tension coefficient of $1/3$. Here, the bathymetry $h$ acts as coupling between the earth and fluid models. Using  finite-element method,  the authors performed numerical simulation to describe the dynamical effects of different kinds of seabed deformations.

In \cite{IKM}, the authors derived, among other models, \eqref{boussinesq_0}  (see (2.34)-(2.35) in \cite{IKM}). Through further approximations, they introduced a new systems capable of modeling a broad range of phenomena, including tsunamis, tidal waves and wave propagation in ports and lakes. In particular, they analyzed the capability of the new models to capture  the shoaling effects over gentle slopes.

 Another contribution \cite{MG}, develops  a BBM-BBM-type regime of \eqref{boussinesq_0}, with parameters $\tilde{a}_1=1/2(2-\sqrt{6})$ and $\tilde{c}_1=1-\sqrt{2/3}$, (see (4)-(5) in \cite{MG}). The authors developed numerical simulations, and they observed that the numerical experiments were in accordance with the previous theoretical and experimental studies. Notably, their results reveal the emergence of the so-called Bragg resonant reflection when surface waves interact with periodic bottom structures. The authors also presented experiments involving the interaction of incident waves with variable topographies such as the shoaling of a solitary wave on a slope and the generation of surface waves by moving topography.

\subsection{Method of proof}

The proof of Theorem \ref{MT1} is based on the introduction of virial identities from which local dispersion in the energy space is clearly identified. Given that \eqref{boussinesq} is a two-component system, any virial term must incorporate contributions from both variables, $u$ and $\eta$. Let us recall the situation in the flat bottom case, described in previous works \cite{KMPP,KM20}. First of all, small data solutions are globally well-defined thanks to the conservation of the Hamiltonian. It turns out that the main virial terms are given by
\[
\mathcal I(t) := \int \vp(t,x) (u\eta + \px u \px\eta)(t,x)dx, \quad \mathcal J(t) := \int \partial_x \vp (t,x)\eta \px u (t,x)dx,
\]
where $\vp$ is a smooth and bounded function. The introduction of the functional $\mathcal I$ is straightforward and directly motivated by considerations of the momentum
\[
P[u,\eta](t) := \int (u\eta + \partial_x u\partial_x\eta)(t,x)dx,
\]
which in the flat bottom case is conserved by the $H^1\times H^1$ flow \cite{BCS2}. The functional $\mathcal J$ can be viewed as a second-order correction to the main functional $\mathcal I(t)$, allowing us to cancel several unfavorable terms arising in the time evolution of $\mathcal I(t)$. Under Definition \ref{New Dis_Par}, in \cite{KMPP,KM19} it was proved that a suitable linear combination of $\mathcal I(t)$ and $\mathcal J(t)$ are enough to show dispersion for any small solution in the energy space: there are parameters $\al\in\R$ and $d_0>0$ such that
\be\label{positivo}
\begin{aligned}
\frac{d}{dt}  \big(\mathcal I(t)  + & \al\mathcal J(t)  \big)  \\
& \gtrsim_{a,b,c,d_0} \int e^{-d_0|x|} (u^2 + \eta^2 + (\partial_x u)^2 +(\partial_x \eta)^2)(t,x)dx.
\end{aligned}
\ee
The choice of parameters in the linear combination $\mathcal I(t) +\al\mathcal J(t)$ is not trivial and one needs to change to ``canonical''  variables (see \cite{ElDika2005-2,ElDika_Martel} for its use in the BBM equation). These new variables are introduced specifically to demonstrate positivity of the bilinear form associated with the virial-type evolution, in full analogy with equation \eqref{positivo}. Canonical variables are particularly useful when dealing with local and nonlocal terms of the form 
\[
\int e^{-d_0|x|} u^2 , \quad \int e^{-d_0|x|} u (1-\px^2)^{-1}u,
\]
and similar others, and which in principle have no clear connection or comparison. For a similar use of canonical variables, see \cite{MaMu0}.

In the variable bottom case, the previously mentioned strategy needs to be modified and adapted to the lack of conserved Hamiltonian and momentum laws. New terms involving the time-dependent variable bottom $h$ will appear in the variation of the Hamiltonian, as stated in \eqref{Energy_new_00}. Similarly, the momentum law will be strongly perturbed and, unless one considers hypotheses such as \eqref{hyp_h}, it is not clear that suitable virial will control de decay of small solutions. It is worth to mention that the variable bottom acts, at first order, as a new “potential'', or “solitary wave'' correction to the virial estimates, in the same spirit as solitary wave asymptotic stability estimates are considered. Therefore, the problem of decay in variable bottom is in some sense equivalent to prove asymptotic stability in a suitable regime depending on the assumptions made on the variable bottom. In the $abcd$ case, the asymptotic stability of its solitary waves is far from obvious, and as far as we understand, only a first result \cite{MaMu3}, concerning the KBK $abcd$ model ($b=d=c=0$, $a=-1$) has been considered. See also \cite{MaMu1,MaMu2} for a proof of asymptotic stability in the related fourth order $\phi^4$ and Good Boussinesq models.

The virial technique that we use in this paper is inspired in the recent results \cite{KMPP,KM19}, which are motivated by works by Martel and Merle \cite{MM,MM1}. An important improvement was the extension of virial identities of this type to the case of wave-like equations, see the works by Kowalczyk, Martel and the second author \cite{KMM1,KMM2}. Then, this approach was succesfuly extended to the case of the $abcd$ Boussinesq system in \cite{KMPP}. In this paper, we have followed some of the ideas in \cite{KMPP}, especially the introduction of the functional $\mathcal I(t)$. However, the main difference between \cite{KMPP} and this paper is that here we need to introduce additional modifications of the main virial terms in order to show positivity. We are also able to prove Theorem \ref{MT1} without any assumption on the parity of the data, leading to a general result in the energy space.
For additional results on decay of small and large solutions, see \cite{FLMP,MMPP1,MMPP2,MaMu0}. For the long time behavior of Boussinesq solitary waves using virial techniques, including asymptotic stability, see \cite{Mau,MaMu1}.

\subsection*{Organization of this paper}
This paper is organized as follows: Section \ref{2} contains some preliminary results needed along this paper.  Section \ref{3} deals with the Virial identities needed for the proof of Theorem \ref{MT1}, the positivity of the bilinear form appearing from the virial dynamics, to obtain integrability in time of the local $H^1\times H^1$ norm of the solution, the main consequence of the virial identity. Section \ref{ENERGY} contains energy estimates needed to improve virial estimates from previous sections. Finally, in Section \ref{7} we prove Theorem \ref{MT1}.

\subsection*{Acknowledgments} 
This work is part of the PANDA project (\url{https://team.inria.fr/panda/}). The authors thank Jean-Claude Saut, Andr\'e de Laire and Olivier Goubet for many interesting comments and suggestion to a first version of this paper. Ch. M. would like to thank DIM-CMM at U. Chile where part of this work was done. Cl. Mu. and F. P. would like to thank Inria Lille and Universit\'e de Lille, France, where part of this work was written. 

\section{Preliminaries}\label{2}

This section presents auxiliary results that will be required in the subsequent analysis.

\subsection{Local and global existence} Under \eqref{Conds0}-\eqref{Conds2}, in the flat bottom case one has that local and global well-posedness for the case of small data is direct \cite{BCS1}, thanks to the conservation of the energy \eqref{Energy}. In the case of variable bottom, the situation is not direct at all. However, given the hypotheses on $h$ \eqref{hyp_h}, which consider smallness and boundedness of space and time derivatives, and seen this as a forcing term which integrates in time, one can easily see that local well-posedness is satisfied. As for the global well-posedness, we assume the validity of \eqref{hyp_h} for all nonnegative times, leading to the following first estimates.

\begin{lemma}
Recall the modified energy in \eqref{Energy_new}. Assume \eqref{hyp_h}. Then one has
\begin{enumerate}
\item There exists $c_0>0$ such that if for some $t$ one has $\| (\eta,u)(t)\|_{H^1\times H^1} <C_0 \varepsilon$, then
\be\label{coercividad}
\begin{aligned}
& c_0 (1- C\varepsilon -C_0\varepsilon)\| (\eta,u)(t)\|_{H^1\times H^1}^2 \\
& \qquad \leq H_h[\eta,u ](t) \leq \frac1{c_0}(1+ C\varepsilon  +C_0\varepsilon) \| (\eta,u)(t)\|_{H^1\times H^1}^2.
\end{aligned}
\ee
\item Variation of energy:
\be\label{Energy_new_1}
\begin{aligned}
\frac{d}{dt} H_h[u ,\eta ](t) = &~{} -a c_1 \int  u  \partial_t^2 \partial_x h \\
&~{} +c_1 \int \left( 1+a    +   \eta + h  \right)u (1- \partial_x^2)^{-1}  \partial_t^2 \partial_x h \\
&~{}  + c\int \eta  \partial_th + c a_1 \int \partial_x \eta \partial_x \partial_th \\
&~{}   +(a_1- 1) \int \left(  (1+c) \eta  + \frac12 u^2 \right)(1- \partial_x^2)^{-1}  \partial_th  \\
&~{}  -a_1 \int \left(  (1+c) \eta  + \frac12 u^2 \right) \partial_th  + \frac12 \int u^2  \partial_t h.
\end{aligned}
\ee
\end{enumerate}
\end{lemma}

\begin{proof}
First of all, \eqref{coercividad} is direct from the definition of $H_h$ in \eqref{Energy_new} and the smallness conditions for $h$ \eqref{hyp_h} and $\| (\eta,u)(t)\|_{H^1\times H^1} <C_0 \varepsilon$ (implying $L^\infty$ smallness as well). The proof of \eqref{Energy_new_1} is direct: one has from \eqref{Energy_new},
\[
\begin{aligned}
\frac{d}{dt} H_h[\eta,u ](t) = &~{} \int \left( -a \partial_{x} u \partial_{x} \partial_{t} u  -c  \partial_x \eta \partial_{x} \partial_{t} \eta  + u \partial_t u + \eta \partial_t \eta  +  \frac12 \partial_t(u^2(\eta+ h)) \right)\\
= &~{} \int \left( a   \partial_{x}^2 u  + u   +  u  (\eta+ h)  \right) \partial_t u \\
&~{} + \int \left( c \partial_{x}^2  \eta  +  \eta  + \frac12 u^2 \right)\partial_{t} \eta   + \frac12 \int u^2  \partial_t h.
\end{aligned}
\]
Using \eqref{boussinesq},
\[
\begin{aligned}
\frac{d}{dt} H_h[\eta,u ](t) = &~{} -\int \left( a   \partial_{x}^2 u  + u   +  u  (\eta+ h)  \right) (1- \partial_x^2)^{-1} \partial_x \! \left( c\, \partial_x^2 \eta + \eta  + \frac12 u^2 \right) \\
&~{} + c_1 \int \left( a   \partial_{x}^2 u  + u   +  u  (\eta +h)  \right) (1- \partial_x^2)^{-1}  \partial_t^2 \partial_x h \\
&~{} - \int \left( c \partial_{x}^2  \eta  +  \eta  + \frac12 u^2 \right) (1- \partial_x^2)^{-1} \partial_x \! \left( a\, \partial_x^2 u +u + u (\eta +h) \right) \\
&~{} + \int \left( c \partial_{x}^2  \eta  +  \eta  + \frac12 u^2 \right)(1- \partial_x^2)^{-1} (-1 +a_1 \partial_x^2) \partial_th   + \frac12 \int u^2  \partial_t h.
\end{aligned}
\]
Simplifying, we obtain
\[
\begin{aligned}
\frac{d}{dt} H_h[\eta,u ](t) = &~{} c_1 \int \left( a   \partial_{x}^2 u  + u   +  u  (\eta + h)  \right) (1- \partial_x^2)^{-1}  \partial_t^2 \partial_x h \\
&~{} + \int \left( c \partial_{x}^2  \eta  +  \eta  + \frac12 u^2 \right)(1- \partial_x^2)^{-1} (-1 +a_1 \partial_x^2) \partial_th  \\
&~{}  + \frac12 \int u^2  \partial_t h.
\end{aligned}
\] 
A further simplification gives
\[
\begin{aligned}
\frac{d}{dt} H_h[\eta,u ](t) = &~{} -a c_1 \int  u  \partial_t^2 \partial_x h \\
&~{} +c_1 \int \left( 1+a    +   \eta + h  \right)u (1- \partial_x^2)^{-1}  \partial_t^2 \partial_x h \\
&~{}  -c\int \eta (-1 +a_1 \partial_x^2) \partial_th \\
&~{}   +(a_1- 1) \int \left(  (1+c) \eta  + \frac12 u^2 \right)(1- \partial_x^2)^{-1}  \partial_th  \\
&~{}  -a_1 \int \left(  (1+c) \eta  + \frac12 u^2 \right) \partial_th  + \frac12 \int u^2  \partial_t h.
\end{aligned}
\] 
Now we integrate by parts to get
\[
\begin{aligned}
\frac{d}{dt} H_h[\eta,u ](t) = &~{} -a c_1 \int  u  \partial_t^2 \partial_x h \\
&~{} +c_1 \int \left( 1+a    +   \eta + h  \right)u (1- \partial_x^2)^{-1}  \partial_t^2 \partial_x h \\
&~{}  + c\int \eta  \partial_th + c a_1 \int \partial_x \eta \partial_x \partial_th \\
&~{}   +(a_1- 1) \int \left(  (1+c) \eta  + \frac12 u^2 \right)(1- \partial_x^2)^{-1}  \partial_th  \\
&~{}  -a_1 \int \left(  (1+c) \eta  + \frac12 u^2 \right) \partial_th  + \frac12 \int u^2  \partial_t h.
\end{aligned}
\]
This proves \eqref{Energy_new_1}, as desired.
\end{proof}

\subsection{Proof of global existence} In this subsection, we prove the first part of Theorem \ref{MT1}, namely \eqref{global}.

\begin{lemma}\label{Lem2p3}
One has
\be\label{est_E}
\begin{aligned}
\left| \frac{d}{dt} H_h[\eta,u ](t) \right| \lesssim  
&~{}  (\| (\eta,u)(t)\|_{H^1\times H^1} + \| (\eta,u)(t)\|_{H^1\times H^1}^2 ) \| \partial_t h (t) \|_{H^1} \\
&~{} + ( \| (\eta,u)(t)\|_{H^1\times H^1}  + \|(\eta,u)(t)\|_{H^1\times H^1}^2 ) \| \partial_t^2 h (t) \|_{L^2},
\end{aligned}
\ee
and
\be\label{est_E_2}
\begin{aligned}
& (1 -C\varepsilon -C_0\varepsilon) \|(\eta,u)(t)\|_{H^1\times H^1}^2 \\
&~{} \le C \|(\eta,u)(0)\|_{H^1\times H^1}^2 \\
&~{} \quad + C\int_0^t  (\| (\eta,u)(s)\|_{H^1\times H^1} + \| (\eta,u)(s)\|_{H^1\times H^1}^2 ) ( \| \partial_t h (s) \|_{H^1} +\| \partial_t^2 h (s) \|_{L^2} )ds.
\end{aligned}
\ee
\end{lemma}

\begin{proof}
From \eqref{Energy_new_1} and \eqref{hyp_h},
\[
\begin{aligned}
& \left| -a c_1 \int  u  \partial_t^2 \partial_x h  + c\int \eta  \partial_th + c a_1 \int \partial_x \eta \partial_x \partial_th \right| 
\\
&\quad  \lesssim \| \partial_x u(t)\|_{L^2} \| \partial_t^2 h (t) \|_{L^2} +\| \eta(t)\|_{L^2} \| \partial_t h (t) \|_{L^2} +\| \partial_x \eta (t)\|_{L^2} \| \partial_x \partial_t h (t) \|_{L^2} \\
&\quad  \lesssim \| (\eta,u)(t)\|_{H^1\times H^1} \left( \| \partial_t h (t) \|_{H^1} + \| \partial_t^2 h (t) \|_{L^2} \right).
\end{aligned}
\]
Similarly,
\[
\begin{aligned}
& \left| -a_1 \int \left(  (1+c) \eta  + \frac12 u^2 \right) \partial_th  + \frac12 \int u^2  \partial_t h\right| \\
& \quad \lesssim \| (\eta,u)(t)\|_{H^1\times H^1}  \| \partial_t h (t) \|_{L^2} +\| (\eta,u)(t)\|_{H^1\times H^1}^2  \| \partial_t h (t) \|_{L^2}.
\end{aligned}
\]
Now, using that $\| (1-\partial_x^2)^{-1} f \|_{L^2} \leq \|f\|_{L^2}$,
\[
\begin{aligned}
&\left| c_1 \int \left( 1+a    +   \eta + h  \right)u (1- \partial_x^2)^{-1}  \partial_t^2 \partial_x h \right| \\
&\quad \lesssim (1+ \|h(t)\|_{L^\infty} + \|\eta(t)\|_{L^\infty}) \| \partial_x u(t)\|_{L^2} \|(1- \partial_x^2)^{-1} \partial_t^2 h (t) \|_{L^2} \\
&\quad \lesssim (1+ \|h(t)\|_{H^1} +\|(\eta,u)(t)\|_{H^1\times H^1}  ) \|(\eta,u)(t)\|_{H^1\times H^1} \| \partial_t^2 h (t) \|_{L^2} .
\end{aligned}
\]
Finally,
\[
\begin{aligned}
&\left|  (a_1- 1) \int \left(  (1+c) \eta  + \frac12 u^2 \right)(1- \partial_x^2)^{-1}  \partial_th \right| \\
& \quad \lesssim \| (\eta,u)(t)\|_{H^1\times H^1}  \| \partial_t h (t) \|_{L^2} +\| (\eta,u)(t)\|_{H^1\times H^1}^2  \| \partial_t h (t) \|_{L^2}.  
\end{aligned}
\]
Gathering the previous estimates and using \eqref{hyp_h}, we get
\[
\begin{aligned}
\left| \frac{d}{dt} H_h[\eta,u ](t) \right| \lesssim  
&~{}  (\| (\eta,u)(t)\|_{H^1\times H^1} + \| (\eta,u)(t)\|_{H^1\times H^1}^2 ) \| \partial_t h (t) \|_{H^1} \\
&~{} + ( \| (\eta,u)(t)\|_{H^1\times H^1}  + \|(\eta,u)(t)\|_{H^1\times H^1}^2 ) \| \partial_t^2 h (t) \|_{L^2} \\
&~{}  + \|h(t)\|_{H^1} \|(\eta,u)(t)\|_{H^1\times H^1} \| \partial_t^2 h (t) \|_{L^2} .
\end{aligned}
\]
Using \eqref{hyp_h}, this proves \eqref{est_E}. Then \eqref{est_E_2} follows directly from \eqref{est_E}, \eqref{coercividad} and integration in time.  and ends the proof of Lemma \ref{Lem2p3}.
\end{proof}
Now we conclude the global well-posedness result in the case of small data, in the positive time case. The negative time case is completely analogous. Let $\varepsilon_0>0$ be a small parameter and assume $0<\varepsilon<\varepsilon_0$. Take initial data $\| (\eta,u)(t=0)\|_{H^1\times H^1} < \varepsilon$, assume \eqref{hyp_h} for such $\varepsilon$ and define, for $C_0>1$,
\begin{equation}\label{TC}
T_*(C_0):= \sup\left\{ T \in (0, +\infty) ~ : ~ \hbox{ for all $t\in [0,T]$, } ~  \| (\eta,u)(t)\|_{H^1\times H^1} <C_0 \varepsilon \right\}.
\end{equation}
Let us prove that if $C_0$ is large enough but fixed, and $\varepsilon_0$ is sufficiently small, then $T_*=+\infty$ in \eqref{TC}. Indeed, let us assume $T_* <+\infty$. Then, assuming that the solution is not trivial, we have from \eqref{est_E_2}
\[
\begin{aligned}
& (1-C\varepsilon -C_0\varepsilon) \|(\eta,u)(t)\|_{H^1\times H^1}^2 \\
&~{} \le C \varepsilon^2  + C( C_0 \varepsilon +C_0^2 \varepsilon^2)  \int_0^{T_*}  ( \| \partial_t h (s) \|_{H^1} +\| \partial_t^2 h (s) \|_{L^2} )ds.
\end{aligned}
\]
and then, using \eqref{hyp_h}, $C\varepsilon <\frac14$ and $C_0\varepsilon<\frac14$,
\[
\begin{aligned}
& \sup_{t\in [0,T_*]} \|(\eta,u)(t)\|_{H^1\times H^1}^2  \le C \varepsilon^2 + C(C_0 \varepsilon + C_0^2 \varepsilon^2) \varepsilon \leq \frac14 C_0^2 \varepsilon^2.
\end{aligned}
\]
This contradicts the maximality of $T_*$ and proves the global well-posedness under small data. This proves \eqref{global}.

\subsection{Review on canonical variables} We mention now previous results by El Dika and Martel \cite{ElDika2005-2,ElDika_Martel}, and further estimates proved in \cite{KMPP} on the BBM \cite{BBM} character of $b=d>0$ abcd waves. These are based on the fact that virial estimates are not well-suited for the variables $(u,\eta)$, but instead better suited for natural canonical variables that appear in models strongly involving the nonlocal operator $(1-\px^2)^{-1}$.

\begin{definition}[Canonical variable]\label{Can_Var}
Let $u=u(x) \in L^2$ be a fixed function. We say that $f$ is canonical variable for $u$ if $f$ uniquely solves the elliptic equation
\be\label{Canonical}
f- \partial_x^2 f  = u, \quad f\in H^2(\R).
\ee
In this case, we simply denote $f=  (1-\partial_x^2)^{-1} u.$
\end{definition}

Canonical variables are standard in equations where the operator $(1-\partial_x^2)^{-1}$ appears; one of the well-known example is given by the Benjamin-Bona-Mahoney BBM equation, see e.g. \cite{ElDika2005-2,ElDika_Martel}.

\begin{lemma}
One has
\begin{equation}\label{eq:L2}
\int \vp' u^2 = \int\vp'\left(f^2 + 2(\partial_x f )^2 + (\partial_x^2 f )^2\right) - \int \vp'''f^2,
\end{equation}
and
\begin{equation}\label{eq:nonlocal}
\int \vp' u \nlop u = \int\vp'\left(f^2 + (\partial_x f )^2\right) - \frac12\int \vp'''f^2.
\end{equation}
\end{lemma}

\begin{lemma}[Equivalence of local $L^2$ and $H^1$ norms, \cite{KMPP}]\label{lem:L2 comparable}
Let $\phi$ be a smooth, bounded positive weight satisfying $|\phi''| \le \lambda \phi$ for some small but fixed $0 < \lambda \ll1$.  Let $f$ be a canonical variable for $u$, as introduced in Definition \ref{Can_Var} and \eqref{Canonical}.  The following are satisfied:

\begin{itemize}
\item If $u\in L^2$, then for any $a_1,a_2,a_3 > 0$, there  exist $c_1, C_1 >0$, depending on $a_j$ ($j=1,2,3$) and $\lambda >0$, such that
\begin{equation}\label{eq:L2_est}
c_1  \int \phi \, u^2 \le \int \phi\left(a_1f^2+a_2(\partial_x f )^2+a_3(\partial_x^2 f )^2\right) \le C_1 \int \phi \, u^2.
\end{equation}
\item If $u\in H^1$, then for any $d_1,d_2,d_3 > 0$, there  exist $c_1, C_1 >0$ depending on $d_j$, $j=1,2,3$, and $\lambda >0$ such that
\begin{equation}\label{eq:H1_est}
c_2 \int \phi (\partial_x u )^2 \le  \int \phi\left(d_1(\partial_x f )^2+d_2(\partial_x^2 f )^2+d_3(\partial_x^3 f)^2\right) \le C_2 \int \phi (\partial_x u)^2.
\end{equation}
\end{itemize}
\end{lemma}

The following results are well-known in the literature, see El Dika \cite{ElDika2005-2} for further details and proofs.

\begin{lemma}[\cite{ElDika2005-2,KMPP}]
The operator $ (1-\px^2)^{-1}$ satisfies the following properties:
\begin{itemize}
\item Suppose that $\phi =\phi(x)$ is such that
\[
(1-\px^2)^{-1}\phi(x) \lesssim \phi(x), \quad x\in \R,
\]
for $\phi(x) > 0$ satisfying $|\phi^{(n)}(x)| \lesssim \phi(x)$, $n \ge 0$. Then, for $v,w,h \in H^1$, we have
\begin{equation}\label{eq:nonlinear1-1}
\int \phi^{(n)} v (1-\px^2)^{-1} \partial_x (wh) ~\lesssim ~ \norm{v}_{H^1} \int \phi (w^2 + ( \partial_x w)^2 +h^2 + ( \partial_x h)^2),
\end{equation}
and
\begin{equation}\label{eq:nonlinear1-2}
\int\phi^{(n)} v (1-\px^2)^{-1}(wh) ~\lesssim ~\norm{v}_{H^1} \int \phi(w^2 +h^2).
\end{equation}
\item Under the previous conditions, we have
\begin{equation}\label{eq:nonlinear1-3}
\int  \partial_x (\phi  ~{}\partial_x v) (1-\px^2)^{-1}(wh) \lesssim \norm{v}_{H^1} \int \phi (w^2 + ( \partial_x w)^2 +h^2 + ( \partial_x h)^2).
\end{equation}
and
\begin{equation}\label{eq:nonlinear1-4}
\int \phi  \partial_x v (1-\px^2)^{-1}  \partial_x(wh) \lesssim \norm{v}_{H^1} \int \phi (w^2 + ( \partial_x w)^2 +h^2 + ( \partial_x h)^2).
\end{equation}
\end{itemize}
\end{lemma}

The following lemmas, proved in \cite{KMM}, will be useful to estimates the terms on \eqref{eq:gvirial} which involves  the influence of the bottom $h$  and the nonlocal operator  $(1-\  \partial_x^2)^{-1}$.
This operator satisfies:
\begin{lemma}[\cite{KMM}]\label{lem:est_IOp}
	Let $f\in L^2(\R)$ and $0<\ <1$ fixed. We have the following estimates:
	\begin{enumerate}
		\item $\| (1-\  \partial_x^2)^{-1}f\|_{L^2(\R)}\leq \|f \|_{L^2(\R)}$,
		\vspace{0.1cm}
		\item $\| (1-\  \partial_x^2)^{-1}\partial_x f\|_{L^2(\R)}\leq \|f \|_{L^2(\R)}$,
		\vspace{0.1cm}
		\item $\| (1-\  \partial_x^2)^{-1}f\|_{H^2(\R)}\leq \|f \|_{L^2(\R)}$.
	\end{enumerate}
\end{lemma}
We also enunciate the following results that appeared in \cite{KMMV20,Mau}. Notice that now the variable in the weights will be $K=1/\lambda(t)$, with $\lambda(t)$ time dependent and increasing, but this shift will not affect the final outcome.
\begin{lemma}
	There exist $C>0$ such that for $0<K\leq 1$ and $g\in L^2$, 
	\begin{align}
	\label{eq:sech_Opg}
	\left\| \sech\left( K \cdot\right) (1-\ \partial_x^2)^{-1} g\right\|_{L^2}
	\leq{}~ C &\left\| (1-\ \partial_x^2)^{-1}\left[ \sech\left( K  \cdot \right)  g\right]\right\|_{L^2},
	\\
\label{eq:sech_Opg_p}
	\left\| \sech(K \cdot)(1-\  \partial_x^2)^{-1}\partial_x g \right\|_{L^2} 
	\leq&{}~ C  \left\| \sech(K \cdot) g \right\|_{L^2},
	\end{align}
and
\begin{equation}\label{eq:sech_Opg_1pp}
	\left\| \sech(K\cdot)(1-\  \partial_x^2)^{-1}(1-\partial_x^2)g \right\|_{L^2} \leq {}~ C  \| \sech(K\cdot ) g \|_{L^2}.
	\end{equation}
\end{lemma}

\section{Modified virial functionals}\label{3}

This Section is devoted to the introduction and study of three virial functionals adapted to the uneven bottom, and their behavior under the $H^1\times H^1$ flow.  Let $\vp=\vp(x)$ be a smooth, bounded weight function, to be chosen later. For each $t\in\R$, we consider the following functionals for some $\vp$ (to be chosen later):
\be\label{I}
\mathcal I(t) := \int \vp(x)(u\eta + \px u \px\eta)(t,x)dx,
\ee
and
\be\label{J}
\mathcal J(t) := \int \vp'(x)(\eta \px u)(t,x)dx.
\ee
Clearly each functional above is well-defined for $H^1\times H^1$ functions, as long as the pair $(u,\eta)(t=0)$ is small in the energy space.

\subsection{Virial functional $\mathcal I(t)$} Using \eqref{boussinesq} and integration by parts, we have the following result.

\begin{lemma}\label{Virial_bous}
Consider $\mathcal I(t)$ in \eqref{I}. Then for any $t\geq0$,
\be\label{Virial0}
\begin{aligned}
\frac{d}{dt} \mathcal I(t) = &~  {}   -\frac{a}2 \int  \varphi' (\partial_x u)^2-\frac{c}2 \int  \varphi' (\partial_x \eta)^2 \\
&~{}   - \left(a+\frac12\right)   \int  \varphi' u^2 -\left( c+ \frac12 \right)  \int  \varphi' \eta^2 \\
& ~ {} + (1+ a)\int \varphi' u (1-\partial_x^2)^{-1} u  + (1+ c)\int \varphi' \eta (1-\partial_x^2)^{-1} \eta  \\
&~ {}     -\frac12 \int  \varphi' u^2 \eta   + \int \varphi' u (1-\partial_x^2)^{-1}\left(u \eta \right)  + \frac12\int \varphi' \eta (1-\partial_x^2)^{-1}\left(u^2\right)\\
&~{} - \frac12 \int  \partial_x (\varphi h) u^2  + \int \varphi' u (1-\partial_x^2)^{-1}\left( u h \right) \\
&~ -  c_1 \int \varphi' \eta (1-\partial_x^2)^{-1}  \partial_t^2 \partial^2_x h  +  \int  \varphi' u(1-\partial^2_x)^{-1}(1- a_1\partial_x^2)\partial_t\partial_{x}h\\
&~{} +c_1  \int \varphi \eta \partial_t^2 \partial_x h + \int \varphi u (-1+  a_1 \partial_x^2)\partial_t h.
\end{aligned}
\ee
\end{lemma}

\begin{proof}
We compute:
\[
\begin{aligned}
\frac{d}{dt} \mathcal I(t) =&~  \int \varphi (\partial_t \eta u + \eta \partial_t u+ \partial_t\partial_xu \partial_x \eta +\partial_xu  \partial_t\partial_x \eta )\\
=&~ {} \int \varphi (\partial_t \eta -\partial_t \partial_x^2 \eta) u + \int  \varphi (\partial_t u  - \partial_t \partial_x^2 u) \eta -\int \varphi' \eta \partial_t \partial_x u -\int \varphi' u \partial_t \partial_x\eta \\
=:&~{} I_1+I_2+I_3+I_4.
\end{aligned}
\]
Replacing \eqref{boussinesq}, and integrating by parts, we get
\[
\begin{aligned}
I_1 =&~  \int \partial_x(\varphi  u ) (a \partial_x^2 u +u +u(\eta+ h) ) + \int \varphi u (-1+  a_1 \partial_x^2)\partial_t h,
\end{aligned}
\]
and
\[
\begin{aligned}
I_2 =&~   \int  \partial_x(\varphi \eta) \left(c\partial_x^2 \eta +\eta + \frac12 u^2\right)+c_1  \int \varphi \eta  \partial_t^2\partial_xh.
\end{aligned}
\]
First of all, using \eqref{boussinesq},
\[
\begin{aligned}
I_3= &~  \int \partial_x (\varphi' \eta) \partial_t u =  \int \partial_x^2 (\varphi' \eta) (1-\partial_x^2)^{-1}\left(c\partial_x^2 \eta +\eta +\frac12 u^2\right)\\
&+ c_1 \int \partial_x (\varphi' \eta) (1-\partial_x^2)^{-1}  \partial_t^2 \partial_x h\\
=&~   \int \left( \partial_x^2  (\varphi' \eta) - \varphi' \eta \right) (1-\partial_x^2)^{-1}\left(c \partial_x^2 \eta +\eta +\frac12 u^2\right) \\
&~ {}  + \int \varphi' \eta (1-\partial_x^2)^{-1}\left(c \partial_x^2 \eta +\eta +\frac12 u^2\right) + c_1 \int \partial_x (\varphi' \eta) (1-\partial_x^2)^{-1}  \partial_t^2 \partial_x h \\
=&~  - \int  \varphi' \eta\left(c\partial_x^2 \eta +\eta +\frac12 u^2\right) \\
&~ {}  + \int \varphi' \eta (1-\partial_x^2)^{-1}\left(c\partial_x^2  \eta +\eta +\frac12 u^2\right) +  c_1 \int \partial_x (\varphi' \eta) (1-\partial_x^2)^{-1}  \partial_t^2 \partial_x h .\\ 
\end{aligned}
\]
Therefore,
\begin{equation}\label{I2+I3}
\begin{aligned}
I_2 + I_3 =& \int  \varphi \partial_x \eta \left(c \partial_x^2  \eta +\eta + \frac12 u^2\right)+ \int \varphi' \eta (1-\partial_x^2)^{-1}\left(c\partial_x^2 \eta +\eta +\frac12 u^2\right)\\
&+c_1  \int \varphi \eta \partial_t^2 \partial_x h  + c_1 \int \partial_x (\varphi' \eta) (1-\partial_x^2)^{-1}   \partial_t^2 \partial_x h .
\end{aligned}
\end{equation}
Similarly,
\[
\begin{aligned}
I_4  =&~ -\int \varphi' u \partial_t \partial_x\eta\\
=&~ \int \varphi' u (1- \partial_x^2)^{-1} \partial_x^2\! \left( a\, \partial_x^2 u +u + (\eta+h) u  \right)- \int \varphi' u (1- \partial_x^2)^{-1} \partial_x (-1 + a_1 \partial_x^2) \partial_th   \\
= &~{}  - \int  \varphi' u \left(a \partial_x^2 u +u +u (\eta+ h) \right)   + \int \varphi' u (1-\partial_x^2)^{-1}\left(a \partial_x^2 u + u + u (\eta+ h)\right) \\
&~{}  - \int  \varphi' u(1-\partial^2_x)^{-1}(-1+ a_1\partial^2_x)\partial_{x}\partial_{t} h,
\end{aligned}
\]
and
\begin{equation}\label{I1+I4}
\begin{aligned}
I_1 + I_4 =& \int  \varphi \partial_x u \left(a \partial_x^2 u  +u + u (\eta +h) \right) \\
&~{} + \int \varphi' u (1-\partial_x^2)^{-1}\left(a \partial_x^2 u +u + u (\eta+ h) \right)\\
&~ {} - \int  \varphi' u(1-\partial^2_x)^{-1}(-1+ a_1\partial^2_x)\partial_{x}\partial_{t} h+ \int \varphi u (-1+  a_1 \partial_x^2)\partial_t h.
\end{aligned}
\end{equation}
Gathering \eqref{I2+I3} and \eqref{I1+I4}, we have
\[
\begin{aligned}
\frac{d}{dt}\mathcal I(t) = &~ \int  \varphi \partial_x u \left(a \partial_x^2  u +u +  u (\eta+ h) \right)\\
&~{} + \int \varphi' u (1-\partial_x^2)^{-1}\left(a \partial_x^2  u +u +u (\eta+ h) \right)\\
&~ {} + \int  \varphi \partial_x \eta \left(c\partial_x^2  \eta +\eta + \frac12 u^2\right)+ \int \varphi' \eta (1-\partial_x^2)^{-1}\left(c \partial_x^2 \eta +\eta +\frac12 u^2\right) \\
&~ {}+  c_1 \int \partial_x (\varphi' \eta) (1-\partial_x^2)^{-1}   \partial_t^2 \partial_x h  - \int  \varphi' u(1-\partial^2_x)^{-1}(-1+ a_1\partial_x^2)\partial_t\partial_{x}h \\
&~{} +c_1  \int \varphi \eta \partial_t^2 \partial_x h + \int \varphi u (-1+  a_1 \partial_x^2)\partial_t h\\
= &~ \int  \varphi \partial_x u \left(a \partial_x^2 u +u + u \eta \right)+ \int \varphi' u (1-\partial_x^2)^{-1}\left(a \partial_x^2 u +u + u \eta \right)\\
&~ {} + \int  \varphi \partial_x \eta \left(c \partial_x^2 \eta +\eta + \frac12 u^2\right)+ \int \varphi' \eta (1-\partial_x^2)^{-1}\left(c \partial_x^2 \eta +\eta +\frac12 u^2\right) \\
 &~ + \int  \varphi \partial_x u \left( uh \right)  +  \int \varphi' u (1-\partial_x^2)^{-1}\left( u h \right) \\
&~ + c_1  \int \partial_x (\varphi' \eta) (1-\partial_x^2)^{-1}  \partial_t^2 \partial_x h - \int  \varphi' u(1-\partial^2_x)^{-1}(-1+ a_1\partial_x^2)\partial_t\partial_{x}h\\
&~{} +c_1  \int \varphi \eta \partial_t^2 \partial_x h + \int \varphi u (-1+  a_1 \partial_x^2)\partial_t h\\
= : &~ \tilde I_1 + \tilde I_2 + \tilde I_3 +\tilde I_4+ \tilde I_5 (h).
\end{aligned}
\]
Now we compute each $\tilde I_j$. First,
\[
\tilde I_1 =  -\frac{a}2 \int  \varphi'  (\partial_x u)^2  -\frac12  \int  \varphi' u^2  -\frac12 \int  \varphi' u^2 \eta - \frac12 \int \varphi u^2 \partial_x \eta.
\]
Second,
\[
\tilde I_3=  -\frac{c}2 \int  \varphi' ( \partial_x  \eta)^2 -\frac12 \int  \varphi' \eta^2 + \frac12 \int  \varphi  \partial_x \eta u^2.
\]
Consequently,
\[
\tilde I_1+ \tilde I_3 = -\frac{a}2 \int  \varphi' ( \partial_x u)^2-\frac{c}2 \int  \varphi' ( \partial_x \eta)^2   -\frac12  \int  \varphi' (u^2 +\eta^2)  -\frac12 \int  \varphi' u^2 \eta.
\]
On the other hand,
\[
\begin{aligned}
\tilde I_2 =  &~  \int \varphi' u (1-\partial_x^2)^{-1}\left(a  \partial_x^2  u +u +u \eta \right) \\
= &~ {} a\int \varphi' u (1-\partial_x^2)^{-1} ( \partial_x^2 u -u) \\
&~ {}+ (1+ a)\int \varphi' u (1-\partial_x^2)^{-1} u   + \int \varphi' u (1-\partial_x^2)^{-1}\left(u \eta \right) \\
= &~ {}  - a\int \varphi' u^2 + (1+ a)\int \varphi' u (1-\partial_x^2)^{-1} u   + \int \varphi' u (1-\partial_x^2)^{-1}\left(u \eta \right).
\end{aligned}
\]
Similarly,
\[
\begin{aligned}
\tilde I_4 =  &~  \int \varphi' \eta (1-\partial_x^2)^{-1}\left(c \partial_x^2 \eta +\eta +\frac12 u^2\right) \\
= &~ {} c \int \varphi' \eta (1-\partial_x^2)^{-1} ( \partial_x^2 \eta -\eta ) \\
&~ {}+ (1+ c)\int \varphi' \eta (1-\partial_x^2)^{-1} \eta   +\frac12 \int \varphi' \eta (1-\partial_x^2)^{-1}\left(u^2\right) \\
= &~ {}  - c\int \varphi' \eta^2 + (1+ c)\int \varphi' \eta (1-\partial_x^2)^{-1} \eta   + \frac12\int \varphi' \eta (1-\partial_x^2)^{-1}\left(u^2\right).
\end{aligned}
\]
Finally,
\begin{align*}
  \tilde I_5= &~  \int  \varphi  \partial_x u \left( uh \right) +  \int \varphi' u (1-\partial_x^2)^{-1}\left( u h \right)\\
&~ + c_1 \int  \partial_x (\varphi' \eta) (1-\partial_x^2)^{-1}  \partial_t^2 \partial_x h  - \int  \varphi' u(1-\partial^2_x)^{-1}(-1+ a_1\partial_x^2)\partial_t\partial_{x}h \\
&~{} +c_1  \int \varphi \eta \partial_t^2 \partial_x h + \int \varphi u (-1+  a_1 \partial_x^2)\partial_t h\\
=&~ - \frac12  \int   \partial_x (\varphi h) u^2  +  \int \varphi' u (1-\partial_x^2)^{-1}\left( u h \right) \\
&~ - c_1 \int \varphi' \eta (1-\partial_x^2)^{-1}   \partial_t^2 \partial^2_x h -  \int  \varphi' u(1-\partial^2_x)^{-1}(-1+  a_1\partial_x^2)\partial_t\partial_{x}h\\
&~{} +c_1  \int \varphi \eta \partial_t^2 \partial_x h + \int \varphi u (-1+  a_1 \partial_x^2)\partial_t h .
\end{align*}
We conclude that
\[
\begin{aligned}
\frac{d}{dt} \mathcal I(t) = &~ \tilde I_1 + \tilde I_2 + \tilde I_3 +\tilde I_4+ \tilde I_5\\
=&~  {}  -\frac{a}2 \int  \varphi' ( \partial_x  u)^2-\frac{c}2 \int  \varphi' ( \partial_x \eta)^2   -\frac12  \int  \varphi' (u^2 +\eta^2)  -\frac12 \int  \varphi' u^2 \eta \\
& ~ {} - a\int \varphi' u^2 + (1+ a)\int \varphi' u (1-\partial_x^2)^{-1} u   + \int \varphi' u (1-\partial_x^2)^{-1}\left(u \eta \right) \\
&~ {} - c\int \varphi' \eta^2 + (1+ c)\int \varphi' \eta (1-\partial_x^2)^{-1} \eta   + \frac12\int \varphi' \eta (1-\partial_x^2)^{-1}\left(u^2\right) \\
&~ - \frac12 \int   \partial_x  (\varphi h) u^2  + \int \varphi' u (1-\partial_x^2)^{-1}\left( u h \right) \\
&~ -  c_1 \int \varphi' \eta (1-\partial_x^2)^{-1}   \partial_t^2 \partial^2_x h  - \int  \varphi' u(1-\partial^2_x)^{-1}(-1+ a_1\partial_x^2)\partial_t\partial_{x}h \\
&~{} +c_1  \int \varphi \eta \partial_t^2 \partial_x h + \int \varphi u (-1+  a_1 \partial_x^2)\partial_t h.
\end{aligned}
\]
Finally, rearranging terms,
\[
\begin{aligned}
\frac{d}{dt} \mathcal I(t)=&~  {}  -\frac{a}2 \int  \varphi' ( \partial_x  u)^2-\frac{c}2 \int  \varphi' ( \partial_x \eta)^2   - \left(a+\frac12\right)   \int  \varphi' u^2 -\left( c+ \frac12 \right)  \int  \varphi' \eta^2 \\
& ~ {} + (1+ a)\int \varphi' u (1-\partial_x^2)^{-1} u  + (1+ c)\int \varphi' \eta (1-\partial_x^2)^{-1} \eta  \\
&~ {}     -\frac12 \int  \varphi' u^2 \eta   + \int \varphi' u (1-\partial_x^2)^{-1}\left(u \eta \right)  + \frac12\int \varphi' \eta (1-\partial_x^2)^{-1}\left(u^2\right)\\
&~{}  - \frac12 \int   \partial_x (\varphi h) u^2   +  \int \varphi' u (1-\partial_x^2)^{-1}\left( u h \right) \\
&~ -  c_1 \int \varphi' \eta (1-\partial_x^2)^{-1}   \partial_t^2 \partial^2_x h  - \int  \varphi' u(1-\partial^2_x)^{-1}(-1+ a_1\partial_x^2)\partial_t\partial_{x}h\\
&~{} +c_1  \int \varphi \eta \partial_t^2 \partial_x h + \int \varphi u (-1+  a_1 \partial_x^2)\partial_t h.
\end{aligned}
\]
This last equality proves \eqref{Virial0}.
\end{proof}

\subsection{Virial functional $\mathcal J(t)$}

In what follows, we recall the system \eqref{boussinesq}-\eqref{Conds} written in the equivalent form
\begin{equation}\label{eq:abcd}
\begin{aligned}
\pt \eta  = &~{} a \px u -(1+a)(1-\partial_x^2)^{-1}\px u - (1-\partial_x^2)^{-1}\px(u(\eta + h)) \\
& ~{} +(1-\partial^2_x)^{-1}\left(-1+ a_1 \partial_x^2\right) \partial_t h \\
\pt u  =  &~{} c \px \eta -(1+c)(1-\partial_x^2)^{-1}\px \eta - \frac12(1-\partial_x^2)^{-1}\px(u^2)\\
& ~{}+ c_1 (1-\partial^2_x)^{-1}  \partial_t^2 \partial_x h.
\end{aligned}
\end{equation}

\begin{lemma}\label{lem:J}
Consider the functional $\mathcal J$ from \eqref{J}. Then for any $t \geq0$,
\begin{equation}\label{eq:J-1}
\begin{aligned}
\frac{d}{dt} \mathcal J(t)=&~  (1+c)\int\vp'\eta^2 - c\int\vp' ( \partial_x \eta)^2  -(1+a)\int\vp' u^2 + a \int\vp' ( \partial_x u)^2\\
&-(1+c)\int\vp'\eta\nlop\eta + (1+a)\int\vp' u \nlop u\\
&+(1+a) \int \vp''u \nlop  \partial_x u + \frac{c}{2}\int\vp''' \eta^2\\
&-\frac12\int\vp' u^2\eta -\frac12\int\vp' \eta\nlop \left( u^2 \right)  \\
&+ \int \vp' u\nlop \left( u\eta \right) +\int \vp'' u\nlop  \partial_x ( u\eta )\\
&- \int\vp'u^2 h + \int\vp'u\nlop(u h) + \int \vp''u\nlop  \partial_x(u h)\\
&+ \int\vp'  \partial_x u (1-\partial^2_x)^{-1}\left(-1+ a_1 \partial_x^2\right) \partial_t h+ c_1   \int\varphi'\eta (1-\partial^2_x)^{-1} \partial_t^2 \partial_x^2 h.
\end{aligned}
\end{equation}
\end{lemma}

\begin{proof}
We compute, using \eqref{eq:abcd},
\[\begin{aligned}
\frac{d}{dt} \mathcal J(t) =&\int\vp'\left( \partial_t \eta  \partial_x u +  \partial_x  \partial_t \eta u \right)\\
=&\int \vp'  \partial_x u \bigg(a  \partial_x u - (1+a)(1-\px^2)^{-1}  \partial_x u - (1-\px^2)^{-1} \partial_x (u(\eta-h)) \\
& \qquad \qquad +(1-\partial^2_x)^{-1}\left(-1+ a_1 \partial_x^2\right) \partial_t h \bigg)\\
&+\int\vp' \eta \bigg(c  \partial_x^2 \eta - (1+c)(1-\px^2)^{-1} \partial_x^2 \eta \\
&  \qquad \qquad - \frac12(1-\px^2)^{-1} \partial_x^2( u^2 )+ c_1(1-\partial^2_x)^{-1} \partial_t^2 \partial_x^2 h \bigg)\\
=: & ~J_1 +J_2.
\end{aligned}\]
We first deal with $J_1$. The integration by parts yields
\begin{equation}\label{eq:J1-1}
\begin{aligned}
J_1=& ~{} a \int \vp' ( \partial_x u)^2 - (1+a)\int\vp'  \partial_x u \nlop  \partial_x u\\
&~{}  - \int\vp'  \partial_x u \nlop  \partial_x (u (\eta + h))+ \int \vp'  \partial_x u (1-\partial^2_x)^{-1}\left(-1+ a_1 \partial_x^2\right) \partial_t h.
\end{aligned}
\end{equation}
We use the identity
\begin{equation}\label{eq:trick1}
\int \psi w \nlop  \partial_x ^2 z = -\int \psi  wz + \int \psi  w \nlop z
\end{equation}
for the second and fourth terms, in the right-hand side of \eqref{eq:J1-1}, to obtain
\[
\begin{aligned}
J_1=&~{}-(1+a)\int \vp' u^2 + a \int \vp' ( \partial_x u)^2\\
&+ (1+a)\int\vp' u \nlop u +(1+a)\int\vp''u\nlop  \partial_x u \\
&-\int\vp'u^2 (\eta+ h) + \int\vp'u\nlop(u (\eta + h)) \\
& +\int \vp''u\nlop  \partial_x (u( \eta+ h)) + \int \vp'  \partial_x u (1-\partial^2_x)^{-1}\left(-1+a_1 \partial_x^2\right) \partial_t h.
\end{aligned}
\]
Arranging terms,
\begin{equation}\label{eq:J1-2}
\begin{aligned}
J_1=&~{}-(1+a)\int \vp' u^2 + a \int \vp' ( \partial_x u )^2\\
&+ (1+a)\int\vp' u \nlop u +(1+a)\int\vp''u\nlop  \partial_x u \\
&-\int\vp'u^2\eta + \int\vp'u\nlop(u\eta) +\int \vp''u\nlop  \partial_x (u\eta)\\
& - \int\vp'u^2 h +  \int\vp'u\nlop(uh) + \int \vp''u\nlop  \partial_x (u h) \\
&+ \int \vp'  \partial_x u(1-\partial^2_x)^{-1}\left(-1+ {a}_1 \partial_x^2\right) \partial_t h.
\end{aligned}
\end{equation}
For $J_2$, using the integration by parts and \eqref{eq:trick1} with the identity
$
f\partial_x^2 f  = \frac12 \partial_x^2 (f^2) - (\partial_x f)^2,
$ 
 yields
\begin{equation}\label{eq:J2}
\begin{aligned}
J_2=&~ c\int \vp'\eta \partial_x^2 \eta - (1+c)\int \vp' \eta \nlop \partial_x^2 \eta \\
&-\frac12\int\vp'\eta \nlop \partial_x^2 (u^2) + {c}_1 \int\varphi' \eta (1-\partial^2_x)^{-1} \partial_t^2 \partial_x^2 h\\
=&~ (1+c)\int\vp'\eta^2 - c\int\vp' (\partial_x \eta)^2 + \frac{c}{2}\int\vp'''\eta^2\\
&-(1+c)\int\vp'\eta\nlop\eta\\
&+\frac12\int\vp'u^2\eta-\frac12\int\vp'\eta\nlop(u^2) +  {c}_1 \int\varphi'\eta (1-\partial^2_x)^{-1} \partial_t^2 \partial_x^2 h.
\end{aligned}
\end{equation}
By collecting all \eqref{eq:J1-2} and \eqref{eq:J2}, we have \eqref{eq:J-1}.
\end{proof}

Now we construct a global virial from a linear combination of $\mathcal I(t)$ and $\mathcal  J(t)$. Let $\alpha$ be a real number. We define the modified virial
\be\label{H}
\mathcal H(t):= \mathcal I(t) + \alpha \mathcal J(t).
\ee
Thanks to Lemmas \ref{Virial_bous} and \ref{lem:J}, we obtain the following direct consequence:

\begin{proposition}[Decomposition of $\frac{d}{dt}\mathcal H(t)$, \cite{KMPP}]\label{prop:general virial}
Let $u$ and $\eta$ satisfy \eqref{boussinesq} and $\mathcal H$ as in \eqref{H}. For any $\alpha\in \R$ and any $t \in \R$, we have the decomposition
\begin{equation}\label{eq:gvirial}
\frac{d}{dt}\mathcal H(t) = \mathcal Q(t) + \mathcal{SQ}(t) + \mathcal{NQ}(t) + \mathcal{NH}(t),
\end{equation}
where $\mathcal Q(t) =\mathcal Q[u,\eta](t) $ is the quadratic form
\begin{equation}\label{eq:leading}
\begin{aligned}
\mathcal Q(t) :=& ~\left(\left(1+ c \right)(\alpha-1) + \frac{1}2\right)\int \vp' \eta^2 -c \Big( \alpha + \frac12\Big)\int \vp' (\partial_x \eta)^2 \\
&+\left(\left(1+ a \right)( -\alpha-1) + \frac{1}2\right)\int \vp' u^2 +a \Big(\alpha-\frac12\Big)\int \vp' (\partial_x u)^2 \\
&+\left(1+ c \right)( - \alpha + 1)\int \vp' \eta \nlop \eta\\
&+\left(1+ a\right)(\alpha  +1)\int \vp' u \nlop u ,
\end{aligned}
\end{equation}
$ \mathcal{SQ}(t)$ represents lower order quadratic terms not included in $\mathcal Q(t)$:
\begin{equation}\label{eq:small linear}
\begin{aligned}
 \mathcal{SQ}(t) :=& ~ 
  \alpha\left(1+ a\right)\int\vp'' u \nlop  \partial_x u
+\frac{\alpha c}{2} \int \vp''' \eta^2 ,
\end{aligned}
\end{equation}
$\mathcal{NQ}(t)$ are truly cubic order terms or higher:
\begin{equation}\label{eq:nonlinear}
\begin{aligned}
 \mathcal{NQ}(t) :=& ~ - \frac12(  \alpha + 1)\int\vp' u^2\eta + \frac12( -\alpha + 1)\int \vp' \eta \nlop (u^2)\\
&+ (\alpha +1)\int \vp'  u \nlop (u\eta) + \alpha\int \vp'' u \nlop \partial_x (u\eta),
\end{aligned}
\end{equation}
and $\mathcal{NH}(t)$ are the new terms appearing from the influence of the bottom $h$ in the dynamics:
\begin{equation}\label{eq:H}
\begin{aligned}
 \mathcal{NH}(t) :=& ~ {}  - \frac12 \int  \partial_x (\varphi h) u^2  + (1+\alpha) \int \varphi' u (1-\partial_x^2)^{-1}\left( u h \right)  \\
 &~{}  - \alpha \int\vp'u^2 h 
  +  \alpha  \int \vp'' u\nlop \partial_x (u h) \\
&~ {} + (\alpha- 1) c_1 \int \varphi' \eta (1-\partial_x^2)^{-1}  \partial_t^2 \partial^2_x h  \\
&~{}  +  \int  \varphi' u(1-\partial^2_x)^{-1}(1- a_1\partial_x^2)\partial_t\partial_{x}h \\
&~{}  +   \alpha  \int\vp' \partial_x u (1-\partial^2_x)^{-1}\left(-1+ a_1 \partial_x^2\right) \partial_t h \\
&~{} + c_1  \int \varphi \eta \partial_t^2 \partial_x h + \int \varphi u (-1+  a_1 \partial_x^2)\partial_t h  . 
\end{aligned}
\end{equation}
\end{proposition}

Proposition \ref{prop:general virial} decomposes the evolution of $abcd$ solutions in four different terms. Recall that the terms \eqref{eq:small linear} are usually absorbed by the terms in $\mathcal Q(t)$. The terms in \eqref{eq:nonlinear} are large unless one consider small data, and may lead to either the existence of solitary waves or blow up. Consequently, they must be taken with care. Finally, the new terms \eqref{eq:H} are different to the previous ones and have a strong influence in the dynamics. 

\subsection{Transformation into canonical variables}

We focus on the quadratic form $\mathcal{Q}(t)$ in \eqref{eq:leading}. We introduce canonical variables for $u$ and $\eta$ (see Definition \ref{Can_Var}) as follows:
\begin{equation}\label{eq:fg}
f := \nlop u \quad \mbox{and} \quad g := \nlop \eta.
\end{equation}
Note that, for $u, \eta \in H^1$, one has $f$ and $g$ in $H^3$. A direct calculation shows the following key relationships between $f$ and $u$ (resp. $g$ and $\eta$), see also Lemma \ref{lem:L2 comparable} for a similar statement. Using \eqref{eq:L2}-\eqref{eq:nonlocal}, we can rewrite the quadratic form $\mathcal{Q}(t)$ as follows:
\begin{lemma}\label{lem:leading}
Let $f$ and $g$ be canonical variables of $u$ and $\eta$ as in \eqref{eq:fg}. Consider the quadratic form $\mathcal Q(t)$ described in \eqref{eq:leading}. Then we have
\begin{equation}\label{eq:leading-1}
\begin{aligned}
\mathcal Q(t) =& \int \vp' \Big( A_1 f^2 + A_2 (\partial_x f )^2 + A_3 (\partial_x^2 f )^2 + A_4 (\partial_x^3 f)^2\Big)\\
&+\int \vp' \Big( B_1 g^2 + B_2 (\partial_x g )^2 + B_3 (\partial_x^2 g)^2 + B_4 (\partial_x^3 g)^2\Big)\\
&+\int \vp''' \Big(D_{11}f^2 + D_{12}(\partial_x f )^2 + D_{21}g^2 + D_{22}(\partial_x g )^2\Big),
\end{aligned}
\end{equation}
where
\[
A_1 = B_1 = \frac12>0,
\]
\[
A_2 =  - \alpha -\frac{3a}{2}, \qquad B_2 = \alpha -  - \frac{3c}{2},
\]
\[
A_3 = -(1-a) \alpha  -2a - \frac12, \qquad B_3 = (1-c) \alpha -2c - \frac12,
\]
\[
A_4= a\left(\alpha   - \frac12\right), \qquad B_4= - c\left(  \alpha + \frac12\right),
\]
and
\[
\begin{aligned}
&D_{11} = \frac12(1+a)(  \alpha  + 1) - \frac12 , \qquad D_{12}=-a\left(\alpha   - \frac12\right),\\
&D_{21}= \frac12(1+c)(1-\alpha  ) - \frac12,\qquad D_{22}= c\left(  \alpha + \frac12\right).
\end{aligned}
\]
\end{lemma}

Lemma \ref{lem:leading} was proved in \cite{KMPP,KM19}. The quadratic form \eqref{eq:leading-1} has different values depending on the sign of $\alpha.$ Precisely, the following lemma, proved in \cite{KM19}, shows that $\mathcal Q(t)$ is bounded below by the $H^1\times H^1$ norm of $(\eta,u)$.

\begin{lemma}[Refined positivity of the quadratic form $\mathcal{Q}(t)$]
Let $a, c <0$ satisfy Definition \ref{New Dis_Par}. Then, we have 
\be\label{Q_new_new}
\mathcal Q(t) \gtrsim \frac{1}{\lambda(t)}\int \sech^2\left( \frac{x}{\lambda(t)} \right)  \left(u^2 + \eta^2 + (\partial_x u)^2 + (\partial_x \eta)^2 \right).
\ee
\end{lemma}

Estimate \eqref{Q_new_new} must be confronted now with the influence of the bottom $h$, specified in the term $\mathcal{NH}(t)$ \eqref{eq:H}. 

Let $\varphi$ be a smooth bounded weight function such that  $\varphi'>0$ and $\la=\lambda(t)$ be the time-dependent weight given by
\begin{equation}\label{eq:lambda}
\lambda(t) := \frac{t}{\log t \log^2 (\log t)}, \quad t\geq 11,
\end{equation}
precisely corresponding to the upper limit of the space interval $I(t)$ \eqref{I(t)}. The following result is stated for times $t\geq 11$, but it can be easily stated and proved for corresponding negative times by correspondely changing \eqref{eq:lambda}.

\begin{lemma}
There exists $C>0$ such that one has
\begin{equation}\label{cota_NH}
\begin{aligned}
 \left|  \mathcal{NH}(t)  \right| 
& \le   (4\delta + C\varepsilon ) \int  \varphi'  u^2 +4 \delta  \int  \varphi'  (\partial_x u)^2  +4\delta \int \varphi' \eta^2 +4\delta \int \varphi' (\partial_x \eta)^2\\
&~{} \quad   + C_\delta \|  \partial_t   h  \|_{L^2}^2 + C_\delta \|  \partial_t \partial_x  h  \|_{L^2}^2  + C_\delta \|  \partial_t^2  h  \|_{L^2}^2 + C\| \partial_x h \|_{L^\infty}\\
&~{} \quad + C  \|\partial_t^2  h \|_{L^2} + C  \| \partial_x \partial_t h \|_{L^2} + \frac{C}{t^{3/2}}.
\end{aligned}
\end{equation}
\end{lemma}

\begin{proof}
From \eqref{eq:H} and \eqref{hyp_h} we compute:  
\[
\begin{aligned}
 \left| \frac12 \int  \partial_x (\varphi h) u^2  \right| \lesssim & ~{} \left| \int  \varphi  \partial_x h u^2  \right|  +\left|  \int  \varphi' h u^2  \right|  \\
 \lesssim &~{} \left| \int  \varphi  \partial_x h u^2  \right|  + \varepsilon  \int  \varphi'  u^2\\
  \lesssim &~{} \| \partial_x h \|_{L^\infty}  + \varepsilon  \int  \varphi'  u^2.
\end{aligned}
\]
Similarly,
\[
\begin{aligned}
& \left|  \alpha \int\vp'u^2 h \right| \lesssim \varepsilon \int \varphi' u^2.
\end{aligned}
\]
Second, we use the bound $\| \varphi' (1-\partial_x^2)^{-1} f\|_{L^2} \lesssim \| \varphi' f\|_{L^2}$,
\[
\begin{aligned}
& \left|   \int \varphi' u (1-\partial_x^2)^{-1}\left( u h \right) \right| \\
& \quad \le \delta \int \varphi' u^2 + C_\delta \int  \varphi' |(1-\partial_x^2)^{-1}\left( u h \right)|^2\\
& \quad \le \delta \int \varphi' u^2 + C_\delta \int  \varphi'  ( u h )^2  \le  (\delta +C_\delta  \varepsilon )\int \varphi' u^2. 
\end{aligned}
\]
Now, using the value of $\lambda$ in \eqref{eq:lambda} along with \eqref{eq:sech_Opg} and \eqref{eq:sech_Opg_p},
\[
\begin{aligned}
&  \left|   \alpha  \int \vp''u\nlop \partial_x (u h)  \right| \\
& \lesssim \left|  \int \vp'''u\nlop uh \right| +\left|  \int \vp'' \partial_x u\nlop \partial_x (uh)  \right| \\
& \lesssim \frac1{\lambda^3} \int \vp' |u| |\nlop uh | + \frac1{\lambda^2(t)} \int \vp' |\partial_x u| |\nlop \partial_x (u h)| \\
& \lesssim \frac1{\lambda^3} \left( \int \vp' u^2 + \int \varphi'  u^2 h^2 \right) + \frac1{\lambda^2(t)}\left(  \int \vp'  (\partial_x u)^2 + \int \varphi' (u h)^2 \right)\\
& \lesssim \frac1{ t^{3/2}}.
\end{aligned}
\]
Similar to previous estimates, the bound $\|(1-\partial_x^2)^{-1} f\|_{L^2}\leq \|f\|_{L^2}$ (Lemma \ref{lem:est_IOp}), we get for $\delta>0$ small,
\[
\begin{aligned}
&   \left| \int \varphi' \eta (1-\partial_x^2)^{-1}  \partial_t^2 \partial^2_x h  \right|  \\
& \quad \lesssim    \left| \int \varphi' \eta   \partial_t^2  h  \right| +  \left| \int \varphi' \eta (1-\partial_x^2)^{-1}  \partial_t^2  h  \right| \le    \delta \int \varphi' \eta^2 + C_\delta \|  \partial_t^2  h  \|_{L^2}^2  . 
\end{aligned}
\]
Using again the bound $\|(1-\partial_x^2)^{-1} f\|_{L^2}\leq \|f\|_{L^2}$,
\[
\begin{aligned}
&   \left|  \int  \varphi' u(1-\partial^2_x)^{-1}(1- a_1\partial_x^2)\partial_t\partial_{x}h   \right|  \\
& \quad \lesssim  \left|  \int  \varphi' u(1-\partial^2_x)^{-1}\partial_t\partial_{x}h   \right|  +  \left|  \int  \varphi' u \partial_t\partial_{x}h   \right|  \\
& \quad \le    \delta \int \varphi' u^2 + C_\delta \|  \partial_t \partial_x  h  \|_{L^2}^2  . 
\end{aligned}
\]
Also, by \eqref{eq:sech_Opg_1pp}, we obtain
\[
\begin{aligned}
&  \left|   \alpha  \int\vp' \partial_x u  (1-\partial^2_x)^{-1}\left(-1+ a_1 \partial_x^2\right) \partial_t h \right|   \le   \delta \int \varphi' (\partial_x u)^2 + C_\delta \|  \partial_t  h  \|_{L^2}^2  .
\end{aligned}
\]
Finally,
\[
\begin{aligned}
& \left| c_1  \int \varphi \eta \partial_t^2 \partial_x h + \int \varphi u (-1+  a_1 \partial_x^2)\partial_t h \right| 
\\
& \quad \lesssim  \left|  \int  \partial_x(\varphi \eta) \partial_t^2 h \right| + \left| \int \varphi u \partial_t h \right|  +\left|  \int \partial_x(\varphi u)\partial_x \partial_t h \right| 
\\
& \quad \le  \delta \int \varphi' \eta^2 +\delta \int \varphi' u^2 + C_\delta  \|  \partial_t^2  h  \|_{L^2}^2 +C_\delta  \|  \partial_t\partial_x  h  \|_{L^2}^2 + C\| u\|_{L^2} \|\partial_t h \|_{L^2} 
\\
& \quad + C\| \partial_x \eta \|_{L^2} \|\partial_t^2  h \|_{L^2} + C\| \partial_x u \|_{L^2} \| \partial_x \partial_t h \|_{L^2}  . 
\end{aligned}.
\]
Gathering the previous estimates, we get \eqref{cota_NH}.
\end{proof}

Using \eqref{Q_new_new} and \eqref{cota_NH}, and the smallness of $\varepsilon,$ we obtain for some $C_0>0,$
\be\label{Q_new_new_new}
\begin{aligned}
\mathcal Q(t) \ge &~{}  \frac{1}{C_0\lambda(t)}\int \sech^2\left( \frac{x}{\lambda(t)} \right) \left(u^2 + \eta^2 + (\partial_x u)^2 + (\partial_x \eta)^2 \right)- \frac{C}{t^{3/2}} \\
&~{}  -  C_0 \left( \|  \partial_t   h  \|_{L^2}^2  +  \|  \partial_x \partial_t  h  \|_{L^2}^2    \right)\\
&~{}  -  C_0 \left(  \|  \partial_t^2  h  \|_{L^2}^2 + \| \partial_x h \|_{L^\infty} + \|\partial_t^2  h \|_{L^2} +\| \partial_x \partial_t h \|_{L^2}  \right) .
\end{aligned}
\ee

\begin{proposition}[Decay in time-dependent intervals]\label{prop:virial2}
Let $(a,b,c)$ be \emph{dispersion-like} parameters defined as in Definition \ref{New Dis_Par}. Let $(u,\eta)(t)$ be $H^1 \times H^1$ global solutions to \eqref{boussinesq} such that \eqref{Smallness} holds. Then, we have
\begin{equation}\label{eq:virial2-1}
\int_2^{\infty}\frac{1}{\lambda(t)}\int \sech^2 \left(\frac{x}{\lambda(t)}\right) \left(u^2 + (\px u)^2 + \eta^2 + (\px \eta)^2\right)(t,x) \, dx \,dt\lesssim 1.
\end{equation}
As an immediate consequence, there exists an increasing sequence of time $\{t_n\}$ $(t_n \to \infty$ as $n \to \infty)$ such that
\begin{equation}\label{eq:virial2-2}
\int \sech^2 \left(\frac{x}{\lambda(t_n)}\right) \left(u^2 + (\px u)^2 + \eta^2 + (\px \eta)^2\right)(t_n,x) \; dx \longrightarrow 0 \mbox{ as } n \to \infty.
\end{equation}
\end{proposition}

\begin{proof}
The proof is standard, but we write it for the sake of completeness. We choose, in functionals $\mathcal{I}$ and $\mathcal{J}$ (see \eqref{I}-\eqref{J}),  the time-dependent weight
\begin{equation}\label{eq:vp}
\vp (t,x):= \tanh \left(\frac{x}{\lambda(t)}\right) \quad \mbox{and} \quad \partial_x\vp(t,x) := \frac{1}{\lambda(t)} \sech^2\left(\frac{x}{\lambda(t)}\right),
\end{equation}
with $\la(t)$ given by \eqref{eq:lambda}. Note that
\[
\begin{aligned}
\lambda'(t)= \frac1{\log t \log^2(\log t)} \left( 1 -\frac{1}{\log t}  -\frac{2}{\log t \log(\log t)} \right),
\end{aligned}
\]
then
\[
\frac{\lambda'(t)}{\lambda(t)} = \frac{1}{t}\left(1-\frac{1}{\log t} -\frac{2}{\log t \log(\log t)} \right) \not\in L^1[11,+\infty),
\]
and
\begin{equation}\label{eq:lambda2}
\frac{\lambda'^2(t)}{\lambda(t)} = \frac1{t \log t \log^2(\log t)} \left( 1 -\frac{1}{\log t}  -\frac{2}{\log t \log(\log t)} \right)^2 \le  \frac1{t \log t \log^2(\log t)}.
\end{equation}
From \eqref{eq:lambda2} one realizes that $\lambda'^2(t)/\lambda(t)$ is integrable in $[11,+\infty)$. Indeed, notice that 
\[
\left( \frac{1}{\log(\log t)}\right)'  = -\frac{1}{t \log t \log^2(\log t)}. 
\]
Then, we compute \eqref{eq:gvirial}, but in addition from \eqref{I} and \eqref{J} we will obtain from \eqref{eq:vp} the following three new terms:
\begin{equation}\label{eq:gvirial-I}
-\frac{\lambda'(t)}{\lambda(t)} \int \frac{x}{\lambda(t)}\sech^2\left( \frac{x}{\lambda(t)}\right) \left(u\eta + \px u \px \eta\right),
\end{equation}
and
\begin{equation}\label{eq:gvirial-J}
-\alpha\frac{\lambda'(t)}{\lambda^2(t)} \int\left(1- \frac{2x}{\lambda(t)}\tanh\left( \frac{x}{\lambda(t)}\right)\right)\sech^2\left( \frac{x}{\lambda(t)} \right) \eta \px u.
\end{equation}
Now we use a similar computation as the one performed in \cite{KMPP}, but $\lambda(t)$ now is slightly better. First, thanks to the global well-posedness, 
\[
\begin{aligned}
& \left|\frac{\lambda'(t)}{\lambda(t)} \int \frac{x}{\lambda(t)} \sech^2\left( \frac{x}{\lambda(t)}\right) \left(u\eta + \px u \px \eta\right) \right|
\\
& \lesssim \frac1{4\lambda(t)} \int \sech^2\left( \frac{x}{\lambda(t)}\right) (\eta^2 +(\partial_x \eta)^2 ) + \frac{C \lambda'^2(t)}{\lambda(t)} \sup_{y \in\mathbb R} y^2 \sech^2 y \int (u^2 +(\partial_x u)^2 ) \\
& \lesssim \frac1{4\lambda(t)} \int \sech^2\left( \frac{x}{\lambda(t)}\right) (\eta^2 +(\partial_x \eta)^2 ) + \frac{C \lambda'^2(t)}{\lambda(t)}\\
& \lesssim \frac1{4\lambda(t)} \int \sech^2\left( \frac{x}{\lambda(t)}\right) (\eta^2 +(\partial_x \eta)^2 ) + \frac{C}{t \log t \log^2(\log t)}.
\end{aligned}
\]
Also,
\[
\left| \alpha\frac{\lambda'(t)}{\lambda^2(t)} \int\left(1- \frac{2x}{\lambda(t)}\tanh\left( \frac{x}{\lambda(t)}\right)\right)\sech^2\left( \frac{x}{\lambda(t)} \right) \eta \px u \right| \lesssim \frac1{t^2} \int (\eta^2 + (\px u)^2).
\]
We conclude,
\begin{equation}\label{eq:virial2-3}
\begin{aligned}
|\eqref{eq:gvirial-I}| &+| \eqref{eq:gvirial-J} | \\
\le&~{} \frac{c_0}{\lambda(t)} \int \sech^2\Big( \frac{x}{\lambda(t)}\Big) \left(u^2 + (\px u)^2 + \eta^2 + (\px \eta)^2\right) + \frac{\wt{C}}{t \log t \log^2(\log t)},
\end{aligned}
\end{equation}
for a fixed constant $\wt{C} > 0$. Using this last estimate \eqref{eq:virial2-3} and \eqref{Q_new_new_new}, we obtain 
\begin{equation}\label{eq:virial2-3.1}
\begin{aligned}
& \frac{1}{\lambda(t)}\int \sech^2\Big(\frac{x}{\lambda(t)}\Big) (u^2 +\eta^2 +(\partial_x u)^2 +(\partial_x \eta)^2 ) \\
& \leq C \frac{d}{dt} \mathcal H(t) + \frac{C}{t \log t \log^2(\log t)} \\
& \quad +C \left( \|  \partial_t   h  \|_{L^2}^2  +  \|  \partial_t \partial_x  h  \|_{L^2}^2  + \|  \partial_t^2  h  \|_{L^2}^2 + \| \partial_x h \|_{L^\infty} \right).
\end{aligned}
\end{equation}
Since from \eqref{hyp_h} the left-hand side of \eqref{eq:virial2-3.1} is integrable in $t$ on $[11,\infty)$, but $\la^{-1}$ does not integrate on $[11,\infty)$, one concludes \eqref{eq:virial2-1} in Proposition \ref{prop:virial2}.
\end{proof}

\section{Energy estimates}\label{ENERGY}

To establish \eqref{Conclusion_0} it is necessary to prove that the local $H^1$ norms of $(u,\eta)(t)$ tends  to zero for all sequences $t_n\to +\infty$, rather than just a single sequence. Proving this result requires the use of a new energy estimate. A key element will be \eqref{coercividad}, which states that for small solutions, the $H^1\times H^1$ norm of $(u,\eta)(t)$ squared and the hamiltonian $H_h[u,\eta]$ in \eqref{Energy_new} are equivalent.

 Let $\psi = \psi(x)$ be a smooth, nonnegative and bounded function, to be chosen in the sequel. We consider the localized energy functional defined by
\begin{equation}\label{eq:local energy}
E_\text{loc}(t) = \frac12 \int \psi(t,x) \left( - a(\px u)^2 - c(\px \eta)^2 + u^2 + \eta^2 + u^2(\eta+h)\right)(t,x)dx.
\end{equation}
(Compare with \eqref{Energy_new}.) For the sake of simplicity, we introduce the following notations:
\be\label{ABFG}
\begin{aligned}
&A := a \partial_x^2u+ u +u(\eta+h), \hspace{2em} B: = c\partial_x^2 \eta + \eta + \frac12u^2,\\
&F := (1-\px^2)^{-1} A, \hspace{2.9em} G: = (1-\px^2)^{-1}B.
\end{aligned}
\ee
(In other words, $F$ and $G$ are the canonical variables of $A$ and $B$, respectively.) Notice that \eqref{boussinesq} now becomes
\begin{equation}\label{boussinesqABFG}
\begin{cases}
\partial_t \eta  + \partial_x F = (1-\px^2)^{-1} (-1 + a_1 \partial_x^2) \partial_th  , \\
\partial_t u  + \partial_xG= c_1  (1-\px^2)^{-1} \partial_t^2 \partial_x h.
\end{cases}
\end{equation}
We will also use canonical variables $f$ and $g$ for $u$ and $\eta$, and the following identities:
\begin{equation}\label{eq:E34-2}
\partial_x^3 f \partial_x g = \partial_x( \partial_x^2 f \partial_xg)  - \partial_x^2 f \partial_x^2 g \hspace{1em} \mbox{and} \hspace{1em} \partial_xf \partial_x^3 g = \partial_x (\partial_x f \partial_x^2 g) - \partial_x^2 f \partial_x^2 g,
\end{equation}
and
\be\label{Aux_00}
\partial_x^2 f g = \partial_x(\partial_x f g) - \partial_xf \partial_xg  \hspace{1em} \mbox{and} \hspace{1em} f \partial_x^2 g = \partial_x(f \partial_xg ) - \partial_x f \partial_xg.
\ee
We have

\begin{lemma}[Variation of local energy $E_\text{loc}$]\label{lem:energy1}
Let $u$ and $\eta$ satisfy \eqref{boussinesq}. Let $f$ and $g$ be canonical variables of $u$ and $\eta$ as in \eqref{eq:fg}. Then, the following holds:
\begin{enumerate}
\item Time derivative. We have in \eqref{eq:local energy},
\begin{equation}\label{eq:energy1}
\begin{aligned}
\frac{d}{dt} E_\emph{loc}(t) =&~{} \int \psi' fg + (1-2(a+c)) \int \psi' \partial_x f \partial_x g  \\
&+ (3ac-2(a+c))\int \psi'  \partial_x^2 f \partial_x^2 g + 3ac\int \psi'  \partial_x^3 f\partial_x^3 g\\
&+ \emph{SNL}_0(t) +\emph{SNL}_1(t)+\emph{SNL}_{h}(t) .
\end{aligned}
\end{equation}
\item The small nonlinear parts $\emph{SNL}_j(t)$ are given by
\begin{equation}\label{eq:SNL0}
\emph{SNL}_0(t) := \frac12 \int \partial_t \psi \left( - a(\px u)^2 - c(\px \eta)^2 + u^2 + \eta^2 + u^2(\eta+h)\right),
\end{equation}
and $\text{SNL}_1(t)$ is given in \eqref{eq:E1E2E3E4_new2}

\item The strong (and weak) interaction terms with the bottom $h$ are given by
\begin{equation}\label{eq:SNLh}
\begin{aligned}
\emph{SNL}_h(t) := & ~{}
 \frac12 \int \psi u^2 \partial_t h 
+c_1  \int \psi  F \partial_t^2\partial_x h +   \int \psi  G (-1+a_1\partial_x^2)\partial_t h 
\\&
-c_1  \int  (\psi'' F (1-\partial_x^2)^{-1}\partial_t^2\partial_x h +2\psi'  F (1-\partial_x^2)^{-1}\partial_t^2\partial_x^2 h  ) 
\\&
-  \int  \psi'' G (1-\partial_x^2)^{-1}(-1+a_1\partial_x^2)\partial_t h
\\&
-  2\int \psi'  G (1-\partial_x^2)^{-1}(-1+a_1\partial_x^2)\partial_x\partial_t h 
.
\end{aligned}
\end{equation}
\end{enumerate}
\end{lemma}

\begin{proof}
Now we  will compute the time derivative of $E_\text{loc}(t)$.
Applying the change of variables with the notations $A, B, F,$ and $G$, together with the definition of $\text{SNL}_0(t)$ in \eqref{eq:SNL0}, gives
\[
\begin{aligned}
& \frac{d}{dt} E_\text{loc}(t) 
\\
&~{} =\frac12 \int \partial_t \psi \left( - a(\px u)^2 - c(\px \eta)^2 + u^2 + \eta^2 + u^2(\eta+h)\right)\\
&~{} \quad + \int \psi \left(-a \partial_x  u \partial_x \partial_t u + u \partial_t u -c\partial_x \eta \partial_x \partial_t \eta + \eta \partial_t \eta +u(\eta +h) \partial_t u + \frac12 u^2  (\partial_t \eta +\partial_t h)\right)\\
&~{} = 
\frac12 \int \partial_t \psi \left( - a(\px u)^2 - c(\px \eta)^2 + u^2 + \eta^2 + u^2(\eta+h)\right)
\\&\qquad \quad+ \int \psi\left(\partial_t u A + \partial_t \eta B\right) + a\int \psi' \partial_t u  \partial_x  u + c\int \psi' \partial_t \eta \partial_x \eta +\frac12 \int \psi u^2 \partial_t h.
\end{aligned}
\]
Using \eqref{boussinesqABFG},
\[
\begin{aligned}
& \frac{d}{dt} E_\text{loc}(t) \\
& = 
\frac12 \int \partial_t \psi \left( - a(\px u)^2 - c(\px \eta)^2 + u^2 + \eta^2 + u^2(\eta+h)\right)
\\&~{}
 -\int \psi \left(\partial_x G  (F-\partial_x^2 F ) + \partial_x F  (G-\partial_x^2 G ) \right)
\\&~{}
+  \int \psi \left(  c_1(F-\partial_x^2 F ) (1-\partial_x^2)^{-1}\partial_t^2\partial_x h +   (G-\partial_x^2 G )(1-\partial_x^2)^{-1}(-1+a_1\partial_x^2)\partial_t h \right)
\\
&~{} + a\int \psi' \partial_t u  \partial_x  u + c\int \psi' \partial_t \eta \partial_x \eta +\frac12 \int \psi u^2 \partial_t h \\
& =
\frac12 \int \partial_t \psi \left( - a(\px u)^2 - c(\px \eta)^2 + u^2 + \eta^2 + u^2(\eta+h)\right)
\\&~{}
+\int \psi' FG - \int \psi'\partial_x F \partial_x G  + a\int \psi' \partial_t u \partial_x u + c\int \psi' \partial_t \eta  \partial_x \eta
~{}
+\frac12 \int \psi u^2 \partial_t h
\\&~{}
+  \int \psi \left(  c_1(F-\partial_x^2 F ) (1-\partial_x^2)^{-1}\partial_t^2\partial_x h +   (G-\partial_x^2 G )(1-\partial_x^2)^{-1}(-1+a_1\partial_x^2)\partial_t h \right)
\\
& =: \text{SNL}_0(t) + E_1+E_2+E_3+E_4 +
E_5+E_6
.
\end{aligned}
\]
We first deal with $E_1$ and $E_2$. From \eqref{ABFG} we have
\[
F = a\, \partial_x^2 f + f + (1-\px^2)^{-1}(u(\eta+h)) \quad \mbox{and} \quad G = c\, \partial_x^2 g + g + \frac12(1-\px^2)^{-1}(u^2).
\]
Using \eqref{eq:E34-2} and \eqref{Aux_00}, then a direct calculation  gives
\begin{equation}\label{eq:E1}
\begin{aligned}
& E_1=\int \psi' FG 
\\
&\quad  =  ac \int \psi' \partial_x^2 f \partial_x^2 g -(a+c) \int \psi'\partial_x f \partial_x g   + \int \psi'fg\\
&\qquad -a\int \psi''\partial_x f  g -c\int \psi'' f \partial_x g \\
&\qquad +\frac{a}{2} \int \psi'\partial_x^2 f  (1-\px^2)^{-1}(u^2) +\frac12\int \psi'f(1-\px^2)^{-1}(u^2)\\
&\qquad +c\int \psi'\partial_x^2 g(1-\px^2)^{-1}(u(\eta+h)) + \int \psi'g(1-\px^2)^{-1}(u(\eta+h)) \\
&\qquad +\frac12\int \psi'(1-\px^2)^{-1}(u(\eta+h))(1-\px^2)^{-1}(u^2).
\end{aligned}
\end{equation}
Similarly, 
\begin{equation}\label{eq:E2}
\begin{aligned}
& E_2=\int \psi' \partial_x F \partial_x G  \\
&\quad  = ac \int \psi'\partial_x^3 f\partial_x^3 g -(a+c) \int \psi'\partial_x^2 f  \partial_x^2 g  + \int \psi'\partial_x f \partial_x g \\
&\qquad -a\int \psi''\partial_x^2 f \partial_x g  -c\int \psi''\partial_x f \partial_x^2 g \\
&\qquad +\frac{a}{2} \int \psi'\partial_x^3 f (1-\px^2)^{-1}\partial_x (u^2) +\frac12\int \psi'\partial_x f (1-\px^2)^{-1}\partial_x (u^2)\\
&\qquad +c\int \psi'\partial_x^3 g(1-\px^2)^{-1}\partial_x (u(\eta+h)) + \int \psi'\partial_x g (1-\px^2)^{-1}\partial_x (u(\eta+h)) \\
&\qquad +\frac12\int \psi'(1-\px^2)^{-1}\partial_x (u(\eta+h)) (1-\px^2)^{-1}\partial_x (u^2).
\end{aligned}
\end{equation}
Now we focus on $E_3$ and $E_4$. Using \eqref{eq:abcd} and integration by parts,
\begin{equation}\label{eq:E3}
\begin{aligned}
E_3 =&~{} a\int \psi' \partial_t u \partial_x u\\
=& ~{}  a\int \psi'  \partial_xu \left( c \px \eta -(1+c)(1-\partial_x^2)^{-1}\px \eta \right)\\
& ~{} + a\int \psi'  \partial_xu \left( - \frac12(1-\partial_x^2)^{-1}\px(u^2) +c_1 (1-\partial^2_x)^{-1}  \partial_t^2 \partial_x h\right)\\
=& ~{} ac\int \psi' \partial_xu\partial_x\eta - a(c+1)\int \psi' \partial_x u (1-\px^2)^{-1}\partial_x\eta\\
&+\frac12 a\int \partial_x(\psi' \partial_xu)(1-\px^2)^{-1}(u^2) + ac_1 \int \psi' \partial_xu  (1-\partial^2_x)^{-1}  \partial_t^2 \partial_x h.
\end{aligned}
\end{equation}
Following a similar approach,
\begin{equation}\label{eq:E4}
\begin{aligned}
E_4 =&~{} c\int \psi' \partial_t \eta \partial_x \eta \\
=&~{} c\int \psi'   \partial_x\eta \left( a \px u -(1+a)(1-\partial_x^2)^{-1}\px u \right) \\
&~{} + c\int \psi'   \partial_x\eta \left(  - (1-\partial_x^2)^{-1}\px(u(\eta+h))  + (1-\partial^2_x)^{-1}\left(-1+ a_1 \partial_x^2\right) \partial_t h\right) \\
=& ~{} ac\int \psi' \partial_x\eta \partial_x u - c(a+1) \int \psi' \partial_x \eta (1-\px^2)^{-1} \partial_x u\\
&~{}+ c\int \partial_x(\psi' \partial_x\eta)(1-\px^2)^{-1}(u(\eta+h))  \\
&~{} + c\int \psi'   \partial_x \eta (1-\partial^2_x)^{-1}\left(-1+ a_1 \partial_x^2\right) \partial_t h .
\end{aligned}
\end{equation}
Therefore, considering \eqref{eq:E3}-\eqref{eq:E4} and adding $E_3$ and $E_4$, we get
\begin{equation}\label{eq:E3E4}
\begin{aligned}
E_3+ E_4 =& ~{} 2ac\int \psi' \partial_x\eta \partial_x u\\
&~{} - a(c+1)\int \psi' \partial_x u (1-\px^2)^{-1}\partial_x\eta - c(a+1) \int \psi' \partial_x \eta (1-\px^2)^{-1} \partial_x u\\
&~{} +\frac12 a\int \partial_x(\psi' \partial_xu)(1-\px^2)^{-1}(u^2) \\
&~{}+ c\int \partial_x(\psi' \partial_x\eta)(1-\px^2)^{-1}(u(\eta+h))  \\
&~{} + ac_1 \int \psi' \partial_xu  (1-\partial^2_x)^{-1}  \partial_t^2 \partial_x h\\
&~{} + c\int \psi'   \partial_x \eta (1-\partial^2_x)^{-1}\left(-1+ a_1 \partial_x^2\right) \partial_t h .
\end{aligned}
\end{equation}
We use now canonical variables $f$ and $g$ for $u$ and $\eta$, and the identities \eqref{eq:E34-2}. We get after integration by parts,
\begin{equation}\label{eq:E34-1.1}
\begin{aligned}
& 2ac  \int \psi' \partial_x u \partial_x \eta  \\
&~ {}\quad= 2ac\int\psi'(\partial_x f \partial_x g  + \partial_x^3 f\partial_x^3 g)  -2ac\int\psi'(\partial_x^3 f\partial_x g  + \partial_x f \partial_x^3 g) \\
&~ {}\quad=2ac\int\psi'(\partial_x f \partial_x g  + \partial_x^3 f\partial_x^3 g)   + 4ac\int\psi' \partial_x^2 f\partial_x^2 g \\
&~ {}\quad\quad + 2ac\int \psi'' ( \partial_x^2 f \partial_xg +\partial_x f \partial_x^2 g).
\end{aligned}
\end{equation}
Similarly,
\begin{equation}\label{eq:E34-1.2}
\begin{aligned}
&- a(c+1)\int \psi' \partial_x u (1-\px^2)^{-1}\partial_x \eta
\\
&~{} \quad = -a(c+1)\int\psi'\partial_x f \partial_x g   +a(c+1)\int\psi'\partial_x^3 f\partial_x g \\
&~{} \quad = -a(c+1)\int\psi'\partial_x f \partial_x g  - a(c+1)\int\psi'\partial_x^2 f\partial_x^2 g - a(c+1)\int\psi''\partial_x^2 f\partial_x g.
\end{aligned}
\end{equation}
Finally,
\begin{equation}\label{eq:E34-1.3}
\begin{aligned}
& - c(a+1) \int \psi' \partial_x \eta(1-\px^2)^{-1}\partial_x u 
\\
&~ {} \quad= -c(a+1)\int\psi'\partial_x f \partial_x g  +c(a+1)\int\psi'\partial_x f \partial_x^3 g \\
&~ {} \quad= -c(a+1)\int\psi'\partial_x f \partial_x g -c(a+1)\int\psi'\partial_x^2 f \partial_x^2 g-c(a+1)\int\psi''\partial_x f \partial_x^2 g.
\end{aligned}
\end{equation}
Replacing \eqref{eq:E34-1.1}-\eqref{eq:E34-1.2}-\eqref{eq:E34-1.3} in \eqref{eq:E3E4},
\begin{equation}\label{eq:E34-3}
\begin{aligned}
E_3+ E_4 =&~{} -(a+c)\int \psi'\partial_x f \partial_x g  + (2ac-(a+c))\int\psi'\partial_x^2 f \partial_x^2 g \\
&~{}+ 2ac\int\psi'\partial_x^3 f\partial_x^3 g\\
&~{}+a(c-1)\int\psi''\partial_x^2 f \partial_x g +c(a-1)\int\psi''\partial_x f \partial_x^2 g\\
&~{} +\frac12 a\int \partial_x(\psi' \partial_xu)(1-\px^2)^{-1}(u^2) \\
&~{}+ c\int \partial_x(\psi' \partial_x\eta)(1-\px^2)^{-1}(u(\eta+h))  \\
&~{} + ac_1 \int \psi' \partial_xu  (1-\partial^2_x)^{-1}  \partial_t^2 \partial_x h\\
&~{} + c\int \psi'   \partial_x \eta (1-\partial^2_x)^{-1}\left(-1+ a_1 \partial_x^2\right) \partial_t h .
\end{aligned}
\end{equation}
Now, from \eqref{eq:E1}, \eqref{eq:E2} and  \eqref{eq:E34-3}, we get
\begin{equation}\label{eq:E1E2E3E4}
\begin{aligned}
& E_1+E_2+E_3+E_4\\
&~{}  =  ac \int \psi' \partial_x^2 f \partial_x^2 g -(a+c) \int \psi'\partial_x f \partial_x g   + \int \psi'fg\\
&\qquad -a\int \psi''\partial_x f  g -c\int \psi'' f \partial_x g \\
&\qquad +\frac{a}{2} \int \psi'\partial_x^2 f  (1-\px^2)^{-1}(u^2) +\frac12\int \psi'f(1-\px^2)^{-1}(u^2)\\
&\qquad +c\int \psi'\partial_x^2 g(1-\px^2)^{-1}(u(\eta+h)) + \int \psi'g(1-\px^2)^{-1}(u(\eta+h)) \\
&\qquad +\frac12\int \psi'(1-\px^2)^{-1}(u(\eta+h))(1-\px^2)^{-1}(u^2)\\
&\qquad + ac \int \psi'\partial_x^3 f\partial_x^3 g -(a+c) \int \psi'\partial_x^2 f  \partial_x^2 g  + \int \psi'\partial_x f \partial_x g \\
&\qquad -a\int \psi''\partial_x^2 f \partial_x g  -c\int \psi''\partial_x f \partial_x^2 g \\
&\qquad +\frac{a}{2} \int \psi'\partial_x^3 f (1-\px^2)^{-1}\partial_x (u^2) +\frac12\int \psi'\partial_x f (1-\px^2)^{-1}\partial_x (u^2)\\
&\qquad +c\int \psi'\partial_x^3 g(1-\px^2)^{-1}\partial_x (u(\eta+h)) + \int \psi'\partial_x g (1-\px^2)^{-1}\partial_x (u(\eta+h)) \\
&\qquad +\frac12\int \psi'(1-\px^2)^{-1}\partial_x (u(\eta+h)) (1-\px^2)^{-1}\partial_x (u^2)\\
&\qquad  -(a+c)\int \psi'\partial_x f \partial_x g  + (2ac-(a+c))\int\psi'\partial_x^2 f \partial_x^2 g \\
&\qquad + 2ac\int\psi'\partial_x^3 f\partial_x^3 g+a(c-1)\int\psi''\partial_x^2 f \partial_x g +c(a-1)\int\psi''\partial_x f \partial_x^2 g\\
&\qquad  +\frac12 a\int \partial_x(\psi' \partial_xu)(1-\px^2)^{-1}(u^2)   + c\int \partial_x(\psi' \partial_x\eta)(1-\px^2)^{-1}(u(\eta+h))  \\
&\qquad  + ac_1 \int \psi' \partial_xu  (1-\partial^2_x)^{-1}  \partial_t^2 \partial_x h  + c\int \psi'   \partial_x \eta (1-\partial^2_x)^{-1}\left(-1+ a_1 \partial_x^2\right) \partial_t h .
\end{aligned}
\end{equation}
After some simplifications, \eqref{eq:E1E2E3E4} becomes
\begin{equation}\label{eq:E1E2E3E4_new}
\begin{aligned}
& E_1+E_2+E_3+E_4\\
&~{}  =  (3ac -2(a+c))\int \psi' \partial_x^2 f \partial_x^2 g +(1-2(a+c)) \int \psi'\partial_x f \partial_x g   + \int \psi'fg\\
&\qquad + 3ac \int \psi'\partial_x^3 f\partial_x^3 g    +  \text{SNL}_1(t),
\end{aligned}
\end{equation}
where
\begin{equation}\label{eq:E1E2E3E4_new2}
\begin{aligned}
& \text{SNL}_1(t)\\
&~{}  := a(c-1)\int\psi''\partial_x^2 f \partial_x g +c(a-1)\int\psi''\partial_x f \partial_x^2 g\\
&\qquad -a\int \psi''\partial_x f  g -c\int \psi'' f \partial_x g  -a\int \psi''\partial_x^2 f \partial_x g  -c\int \psi''\partial_x f \partial_x^2 g \\
&\qquad +\frac{a}{2} \int \psi'\partial_x^2 f  (1-\px^2)^{-1}(u^2) +\frac12\int \psi'f(1-\px^2)^{-1}(u^2)\\
&\qquad +c\int \psi'\partial_x^2 g(1-\px^2)^{-1}(u(\eta+h)) + \int \psi'g(1-\px^2)^{-1}(u(\eta+h)) \\
&\qquad +\frac12\int \psi'(1-\px^2)^{-1}(u(\eta+h))(1-\px^2)^{-1}(u^2)\\
&\qquad +\frac{a}{2} \int \psi'\partial_x^3 f (1-\px^2)^{-1}\partial_x (u^2) +\frac12\int \psi'\partial_x f (1-\px^2)^{-1}\partial_x (u^2)\\
&\qquad +c\int \psi'\partial_x^3 g(1-\px^2)^{-1}\partial_x (u(\eta+h)) + \int \psi'\partial_x g (1-\px^2)^{-1}\partial_x (u(\eta+h)) \\
&\qquad +\frac12\int \psi'(1-\px^2)^{-1}\partial_x (u(\eta+h)) (1-\px^2)^{-1}\partial_x (u^2)\\
&\qquad  +\frac12 a\int \partial_x(\psi' \partial_xu)(1-\px^2)^{-1}(u^2)   + c\int \partial_x(\psi' \partial_x\eta)(1-\px^2)^{-1}(u(\eta+h))  \\
&\qquad  + ac_1 \int \psi' \partial_xu  (1-\partial^2_x)^{-1}  \partial_t^2 \partial_x h  + c\int \psi'   \partial_x \eta (1-\partial^2_x)^{-1}\left(-1+ a_1 \partial_x^2\right) \partial_t h .
\end{aligned}
\end{equation}
Now, let us focus on $E_6$. Integrating by parts and rearranging the terms, we have
\[
\begin{aligned}
E_6
=&~{} c_1  \int \psi  F (1-\partial_x^2)^{-1}\partial_t^2\partial_x h 
-c_1  \int F \partial_x^2(\psi  (1-\partial_x^2)^{-1}\partial_t^2\partial_x h ) 
\\&
+   \int \psi  G (1-\partial_x^2)^{-1}(-1+a_1\partial_x^2)\partial_t h 
-  \int G \partial_x^2(\psi  (1-\partial_x^2)^{-1}(-1+a_1\partial_x^2)\partial_t h )
\\
=&~{} c_1  \int \psi  F \partial_t^2\partial_x h +   \int \psi  G (-1+a_1\partial_x^2)\partial_t h 
\\&
-c_1  \int  (\psi'' F (1-\partial_x^2)^{-1}\partial_t^2\partial_x h +2\psi'  F (1-\partial_x^2)^{-1}\partial_t^2\partial_x^2 h  ) 
\\&
-  \int  (\psi'' G (1-\partial_x^2)^{-1}(-1+a_1\partial_x^2)\partial_t h +2\psi'  G (1-\partial_x^2)^{-1}(-1+a_1\partial_x^2)\partial_x\partial_t h )
\end{aligned}
\]
Letting $\text{SNL}_{h}:=E_5+E_6$, we obtain
\begin{equation}\label{E56}
\begin{aligned}
\mbox{SNL}_{h}(t):=&~{} \frac12 \int \psi u^2 \partial_t h 
+c_1  \int \psi  F \partial_t^2\partial_x h +   \int \psi  G (-1+a_1\partial_x^2)\partial_t h 
\\&
-c_1  \int  (\psi'' F (1-\partial_x^2)^{-1}\partial_t^2\partial_x h +2\psi'  F (1-\partial_x^2)^{-1}\partial_t^2\partial_x^2 h  ) 
\\&
-  \int  \psi'' G (1-\partial_x^2)^{-1}(-1+a_1\partial_x^2)\partial_t h 
\\&
-2  \int  \psi'  G (1-\partial_x^2)^{-1}(-1+a_1\partial_x^2)\partial_x\partial_t h 
.
\end{aligned}
\end{equation}
Thus, from \eqref{eq:E1E2E3E4_new}, \eqref{eq:E1E2E3E4_new2}  and \eqref{E56} we have proved \eqref{eq:energy1}. This ends the proof of Lemma \ref{lem:energy1}. 
\end{proof}

\section{Proof of the Theorem \ref{MT1}}\label{7}

Now we finally prove Theorem \ref{MT1}. Let $\la(t)$ be given by \eqref{eq:lambda}. First, we take in \eqref{eq:energy1}:
\be\label{psi_psi}
\psi (t,x):= \sech^4 \left(\frac{x}{\lambda(t)}\right).
\ee
Note that \eqref{psi_psi} satisfies 
\begin{equation}\label{eq:weight3}
|\partial_x \psi(t,x)| \le \frac{4}{\lambda(t)}\psi (t,x) \qquad \mbox{and} \qquad |\partial_x^2 \psi(t,x)| \le \frac{20}{\lambda(t)} |\partial_x \psi(t,x)|.
\end{equation}

\begin{proposition}\label{prop:energy2}
Let $(a,b,c)$ be parameters defined as in Definition \ref{New Dis_Par}. Let $(u,\eta)(t)$ be a small $H^1 \times H^1$ global solution to \eqref{boussinesq} such that \eqref{Smallness} holds. Then, we have
\begin{equation}\label{eq:energy2.1}
\lim_{t \to \infty} \int \sech^4 \left(\frac{x}{\lambda(t)}\right) \left(u^2 + (\px u)^2 + \eta^2 + (\px \eta)^2\right)(t,x) \; dx = 0.
\end{equation}
\end{proposition}

Notice that this result proves Theorem \ref{MT1}.

\begin{proof}[Proof of Proposition \ref{prop:energy2}]
We shall use Lemma \ref{lem:energy1}. From \eqref{eq:energy1} one has
\begin{equation}\label{eq:energy1_new}
\begin{aligned}
\frac{d}{dt} E_\text{loc}(t) =&~{} \int \psi' fg + (1-2(a+c)) \int \psi' \partial_x f \partial_x g  \\
&+ (3ac-2(a+c))\int \psi'  \partial_x^2 f \partial_x^2 g + 3ac\int \psi'  \partial_x^3 f\partial_x^3 g\\
&+ \text{SNL}_0(t) +\text{SNL}_1(t) +\text{SNL}_h(t).
\end{aligned}
\end{equation}
with
\[
\text{SNL}_0(t) := \frac12 \int \partial_t \psi \left( - a(\px u)^2 - c(\px \eta)^2 + u^2 + \eta^2 + u^2(\eta+h)\right),
\]
and $\text{SNL}_1(t)$ is given in \eqref{eq:E1E2E3E4_new2}. Bounding \eqref{eq:SNL0} is easy following \eqref{psi_psi}: one has
\begin{equation}\label{eq:E1E2E3E4_new2_m1}
\begin{aligned}
& |\text{SNL}_0(t)|   \lesssim {\color{red} \frac1{\lambda(t)} } 
\int \psi (|\eta|^2 +|\partial_x \eta|^2 +|u|^2 +|\partial_x u|^2 +|h|^2 ).
\end{aligned}
\end{equation} 
Following \eqref{eq:L2_est}-\eqref{eq:H1_est}, and estimates \eqref{eq:nonlinear1-1}, \eqref{eq:nonlinear1-2}, \eqref{eq:nonlinear1-3} and \eqref{eq:nonlinear1-4}, we bound $|\text{SNL}_1(t)|$ in \eqref{eq:E1E2E3E4_new2} as follows:
\begin{equation}\label{eq:E1E2E3E4_new2_0}
\begin{aligned}
& |\text{SNL}_1(t)| \\
& \quad  \lesssim \frac1{\lambda(t)}\int \psi (|\eta|^2 +|\partial_x \eta|^2 +|u|^2 +|\partial_x u|^2 +|h|^2 +|\partial_t h|^2 +|\partial_x h|^2 +|\partial_t^2 \partial_xh|^2).
\end{aligned}
\end{equation}
Indeed, one has from \eqref{eq:weight3} and \eqref{eq:L2_est},
\begin{equation}\label{eq:E1E2E3E4_new2_1}
\begin{aligned}
&~{}  \Big| a(c-1)\int\psi''\partial_x^2 f \partial_x g +c(a-1)\int\psi''\partial_x f \partial_x^2 g \Big|  \\
& \qquad \lesssim   \frac1{\lambda^2(t)}\int \psi (|\eta|^2  +|u|^2 ).
\end{aligned}
\end{equation}
Second, using again \eqref{eq:weight3} and \eqref{eq:H1_est},
\begin{equation}\label{eq:E1E2E3E4_new2_2}
\begin{aligned}
&~{}  \Big|  -a\int \psi''\partial_x f  g -c\int \psi'' f \partial_x g  -a\int \psi''\partial_x^2 f \partial_x g  -c\int \psi''\partial_x f \partial_x^2 g \Big|  \\
& \qquad \lesssim  \frac1{\lambda^2(t)}\int \psi (|\eta|^2  +|u|^2 ).
\end{aligned}
\end{equation}
Third, using \eqref{eq:nonlinear1-2},
\begin{equation}\label{eq:E1E2E3E4_new2_3}
\begin{aligned}
&~{}  \Big|\frac{a}{2} \int \psi'\partial_x^2 f  (1-\px^2)^{-1}(u^2) +\frac12\int \psi'f(1-\px^2)^{-1}(u^2) \Big|   \lesssim \frac1{\lambda(t)} \|\eta\|_{H^1}\int \psi |u|^2.
\end{aligned}
\end{equation}
Fourth, using again \eqref{eq:nonlinear1-2},
\begin{equation}\label{eq:E1E2E3E4_new2_4}
\begin{aligned}
&~{}  \Big|c\int \psi'\partial_x^2 g(1-\px^2)^{-1}(u(\eta+h)) + \int \psi'g(1-\px^2)^{-1}(u(\eta+h))  \Big|  \\
& \qquad \lesssim \frac1{\lambda(t)} \| u \|_{H^1}\int \psi (|u|^2 +|\eta|^2 + |h|^2). 
\end{aligned}
\end{equation}
Fifth, using again \eqref{eq:nonlinear1-2},
\begin{equation}\label{eq:E1E2E3E4_new2_5}
\begin{aligned}
&~{}  \Big| \frac12\int \psi'(1-\px^2)^{-1}(u(\eta+h))(1-\px^2)^{-1}(u^2) \Big|  \\
& \qquad \lesssim  \frac1{\lambda(t)} \| (1-\px^2)^{-1}(u^2) \|_{H^1}\int \psi (|u|^2 +|\eta|^2 +|h|^2)\\
& \qquad \lesssim  \frac1{\lambda(t)} \| u \|_{L^2}^2 \int \psi (|u|^2 +|\eta|^2 +|h|^2).
\end{aligned}
\end{equation}
Now, using \eqref{eq:nonlinear1-4},
\begin{equation}\label{eq:E1E2E3E4_new2_6}
\begin{aligned}
&~{}  \Big| \frac{a}{2} \int \psi'\partial_x^3 f (1-\px^2)^{-1}\partial_x (u^2) +\frac12\int \psi'\partial_x f (1-\px^2)^{-1}\partial_x (u^2) \Big| \\
& \qquad \lesssim \frac1{\lambda(t)} \| \eta \|_{H^1}\int \psi (|u|^2  + |\partial_x u |^2).
\end{aligned}
\end{equation}
Similarly,
\begin{equation}\label{eq:E1E2E3E4_new2_7}
\begin{aligned}
&~{}  \Big| c\int \psi'\partial_x^3 g(1-\px^2)^{-1}\partial_x (u(\eta+h)) + \int \psi'\partial_x g (1-\px^2)^{-1}\partial_x (u(\eta+h))\Big| \\
& \qquad \lesssim \frac1{\lambda(t)} \| u \|_{H^1}\int \psi (|u|^2  + |\partial_x u |^2 + |\eta |^2  + |\partial_x \eta |^2 +|h|^2  + |\partial_x h |^2).
\end{aligned}
\end{equation}
Similar to \eqref{eq:E1E2E3E4_new2_5}, and using estimates from \cite{KMM} (see Lemma \ref{lem:est_IOp}),
\begin{equation}\label{eq:E1E2E3E4_new2_8}
\begin{aligned}
&~{}  \Big|\frac12\int \psi'(1-\px^2)^{-1}\partial_x (u(\eta+h)) (1-\px^2)^{-1}\partial_x (u^2) \Big| \\
& \qquad \lesssim  \frac1{\lambda(t)} \| (1-\px^2)^{-1} \partial_x (u^2) \|_{H^1}\\
& \qquad \quad \times \int \psi (|u|^2 +|\partial_xu|^2 +|\eta|^2 +|\partial_x\eta |^2+|h|^2+|\partial_xh|^2)\\
& \qquad \lesssim  \frac1{\lambda(t)} \| u \|_{H^1}^2 \int \psi (|u|^2 +|\partial_xu|^2 +|\eta|^2 +|\partial_x\eta |^2+|h|^2+|\partial_xh|^2).
\end{aligned}
\end{equation}
Similar as in \eqref{eq:E1E2E3E4_new2_7},
\begin{equation}\label{eq:E1E2E3E4_new2_9}
\begin{aligned}
&~{}  \Big| c\int \psi'\partial_x^3 g(1-\px^2)^{-1}\partial_x (u(\eta+h)) + \int \psi'\partial_x g (1-\px^2)^{-1}\partial_x (u(\eta+h)) \Big| \\
& \qquad \lesssim \frac1{\lambda(t)} \| u \|_{H^1}\int \psi (|u|^2  + |\partial_x u |^2 + |\eta |^2  + |\partial_x \eta |^2 +|h|^2  + |\partial_x h |^2).
\end{aligned}
\end{equation}
Now we use \eqref{eq:nonlinear1-3},
\begin{equation}\label{eq:E1E2E3E4_new2_10}
\begin{aligned}
&~{}  \Big|\frac12 a\int \partial_x(\psi' \partial_xu)(1-\px^2)^{-1}(u^2)   + c\int \partial_x(\psi' \partial_x\eta)(1-\px^2)^{-1}(u(\eta+h))   \Big| \\
& \qquad \lesssim \frac{1}{\lambda(t)}  (\|u\|_{H^1} +\|\eta\|_{H^1}) \int \psi (|u|^2  + |\partial_x u |^2 +|\eta|^2  + |\partial_x \eta |^2 + |h|^2  + |\partial_x h |^2).
\end{aligned}
\end{equation}
Finally, using H\"older and \cite{KMM}, 
\begin{equation}\label{eq:E1E2E3E4_new2_11}
\begin{aligned}
&~{}  \Big| ac_1 \int \psi' \partial_xu  (1-\partial^2_x)^{-1}  \partial_t^2 \partial_x h  + c\int \psi'   \partial_x \eta (1-\partial^2_x)^{-1}\left(-1+ a_1 \partial_x^2\right) \partial_t h\Big| \\
& \quad \lesssim \frac{1}{\lambda(t)} \int \psi ( |\partial_x u |^2  + |\partial_x \eta |^2 )\\
& \quad \quad+ \frac{1}{\lambda(t)} \int \psi ( |(1-\partial^2_x)^{-1}  \partial_t^2 \partial_xh|^2  + |\partial_t h |^2 + | (1-\partial^2_x)^{-1}\partial_t h|^2)\\
& \quad \lesssim \frac{1}{\lambda(t)}  \int \psi ( |\partial_x u |^2  + |\partial_x \eta |^2 +  |\partial_t^2 \partial_xh|^2  + |\partial_t h |^2 ).
\end{aligned}
\end{equation}
Now, let us deal with bottom's interaction. Recalling $\text{SNL}_h(t)$ given in \eqref{eq:SNLh} along with \eqref{psi_psi}, the terms on the first line satisfy
\[
\begin{aligned}
&\left\vert  \frac12 \int \psi u^2 \partial_t h 
+c_1  \int \psi  F \partial_t^2\partial_x h +   \int \psi  G (-1+a_1\partial_x^2)\partial_t h \right\vert 
\\
&~{} \leq
\|u\|_{H^1}^4 \|\partial_t h \|_{L^2} +\|F\|_{L^2} \|\partial_t^2\partial_x h  \|_{L^2} +\|G\|_{L^2} \|(-1+a_1\partial_x^2)\partial_t h \|_{L^2} 
\\
&~{} \leq
\|u\|_{H^1}^4 \|\partial_t h \|_{L^2} +\|F\|_{L^2} \|\partial_t^2\partial_x h  \|_{L^2} +\|G\|_{L^2} \|\partial_t h \|_{H^1} 
\\
&~{} \lesssim
 \|\partial_t h \|_{L^2} +\|\partial_t^2\partial_x h  \|_{L^2} + \|\partial_t h \|_{H^1} .
\end{aligned}
\]
Then, following the same ideas applied to $\text{SNL}_0$ and $\text{SNL}_1$, we obtain that 
\begin{equation}\label{eq:SNLh_b}
\begin{aligned}
|\text{SNL}_h(t)| \lesssim  & ~{}
 \|\partial_t h \|_{L^2} +\|\partial_t^2\partial_x h  \|_{L^2} + \|\partial_t h \|_{H^1} 
 \\&+\frac{1}{\lambda(t)}  \int \psi (|u|^2+ |\partial_x u |^2+|\eta|  + |\partial_x \eta |^2 +  |\partial_t^2 \partial_xh|^2  + |\partial_t h |^2 ).
\end{aligned}
\end{equation}
Gathering \eqref{eq:E1E2E3E4_new2_1}, \eqref{eq:E1E2E3E4_new2_2}, \eqref{eq:E1E2E3E4_new2_3}, \eqref{eq:E1E2E3E4_new2_4}, \eqref{eq:E1E2E3E4_new2_5}, \eqref{eq:E1E2E3E4_new2_6}, \eqref{eq:E1E2E3E4_new2_7}, \eqref{eq:E1E2E3E4_new2_8}, \eqref{eq:E1E2E3E4_new2_9}, \eqref{eq:E1E2E3E4_new2_10}, \eqref{eq:E1E2E3E4_new2_11} and \eqref{eq:SNLh_b}, we obtain \eqref{eq:E1E2E3E4_new2_0}.

Thanks to \eqref{eq:E1E2E3E4_new2_0} and \eqref{eq:E1E2E3E4_new2_m1}, we conclude in \eqref{eq:energy1_new},
\begin{equation}\label{eq:energy1_2}
\begin{aligned}
&\left| \frac{d}{dt} E_\text{loc}(t)\right| 
\\
& \quad  \lesssim  \frac1{\lambda(t)}\int   \psi  (|\eta|^2 +|\partial_x \eta|^2 +|u|^2 +|\partial_x u|^2 +|h|^2 +|\partial_t h|^2 +|\partial_x h|^2 +|\partial_t^2 \partial_xh|^2)
\\&
\qquad \qquad +\|\partial_t h \|_{L^2} +\|\partial_t^2\partial_x h  \|_{L^2} + \|\partial_t h \|_{H^1} .
\end{aligned}
\end{equation}
Integrating \eqref{eq:energy1_2} on $[t, t_n]$, for $t < t_n$ as in \eqref{eq:virial2-2}, \eqref{hyp_h} and Proposition \ref{prop:virial2} yield
\[
\begin{aligned}
& \left|E_\text{loc}(t_n) - E_\text{loc}(t) \right| \\
& \quad \lesssim \int_t^{\infty}\!\!   \frac{1}{\lambda(t)}  \int \sech^2 \left(\frac{x}{\lambda(t)}\right) (u^2 + (\px u)^2 + \eta^2 + (\px \eta)^2)(t,x) \; dxdt 
\\
& \quad \quad + \int_t^{\infty}\!\!   \frac{1}{\lambda(t)}  \int \sech^2 \left(\frac{x}{\lambda(t)}\right) (h^2 + (\partial_t h)^2 +(\partial_x h)^2 +(\partial_t^2 \partial_xh)^2)(t,x) \; dxdt 
\\
&\qquad \quad +\int_{t}^{\infty} (\|\partial_t h \|_{L^2} +\|\partial_t^2\partial_x h  \|_{L^2} + \|\partial_t h \|_{H^1} ) 
 < \infty.
\end{aligned}
\]
Note that, from the Sobolev embedding (with $\norm{\eta}_{H^1} \lesssim \ve$ and $\| h\|_{H^1}\lesssim \varepsilon$), we have
\[
|E_\text{loc}(t_n)| \lesssim \int \psi(u^2 + (\px u)^2 + \eta^2 + (\px \eta)^2) (t_n) \longrightarrow 0, \quad \mbox{ as } n \to \infty,
\]
thanks to \eqref{eq:virial2-2}. Thus, by sending $t_n \to \infty$, we have
\[\begin{aligned}
\int \sech^4\left( \frac{x}{\lambda(t)} \right)& (u^2 + (\px u)^2 + \eta^2 + (\px \eta)^2 + \eta u^2) (t) \\
&\lesssim \int_t^{\infty}  \frac{1}{\lambda(t)}  ~{}
\!\!\int \sech^2 \left(\frac{x}{\lambda(t)}\right) (u^2 + (\px u)^2 + \eta^2 + (\px \eta)^2)
\\&\qquad \quad+\int_{t}^{\infty} (\|\partial_t h \|_{L^2} +\|\partial_t^2\partial_x h  \|_{L^2} + \|\partial_t h \|_{H^1} ) .
\end{aligned}\]
Once again, the Sobolev embedding (with $\norm{\eta}_{H^1} \lesssim \ve$) guarantees that
\[
\begin{aligned}
& \lim_{t \to \infty}\int_{-\lambda(t)}^{\lambda(t)} (u^2 + (\px u)^2 + \eta^2 + (\px \eta)^2) (t)\\
& \quad \lesssim  \lim_{t \to \infty}\int \sech^4\left( \frac{x}{\lambda(t)} \right) (u^2 + (\px u)^2 + \eta^2 + (\px \eta)^2) (t) = 0,
\end{aligned}
\]
which proves \eqref{eq:energy2.1}, and completes the proof of Theorem \ref{MT1}, part \eqref{Conclusion_0}.
\end{proof}


\providecommand{\bysame}{\leavevmode\hbox to3em{\hrulefill}\thinspace}
\providecommand{\MR}{\relax\ifhmode\unskip\space\fi MR }
\providecommand{\MRhref}[2]{%
  \href{http://www.ams.org/mathscinet-getitem?mr=#1}{#2}
}
\providecommand{\href}[2]{#2}

\end{document}